\documentclass[12pt,a4paper]{amsart}
\usepackage{amssymb,amscd,amsthm,amsmath}
\usepackage{amsfonts}
\usepackage[top=35mm, bottom=35mm, left=30mm, right=30mm]{geometry}
\usepackage[colorlinks=true,citecolor=blue]{hyperref}
\usepackage{mathptmx}
\usepackage{eucal}
\usepackage{graphicx}
\usepackage{mathrsfs}
\usepackage{xcolor}
\usepackage[all]{xy}
\usepackage{bm}
\usepackage[pagewise]{lineno}\nolinenumbers

\newtheorem{thm}{Theorem}[section]
\newtheorem{cor}[thm]{Corollary}
\newtheorem{lem}[thm]{Lemma}
\newtheorem{prop}[thm]{Proposition}

\newtheorem{que}[thm]{Question}

\theoremstyle{definition}
\newtheorem{df}[thm]{Definition}
\newtheorem{ex}[thm]{Example}
\newtheorem{rem}[thm]{Remark}

\numberwithin{equation}{section}
\def\CC{\mathcal{C}}
\def\UU{\mathcal{U}}
\def\VV{\mathcal{V}}
\def\diam{\text{\rm diam}}
\def\Leb{\text{\rm Leb}}

\def\UU{\mathcal{U}}
\def\VV{\mathcal{V}}
\def\diam{\text{\rm diam}}
\def\Leb{\text{\rm Leb}} 
\def\logf{\log\frac{1}{\epsilon}}

\makeatletter

\newcommand{\Rmnum}[1]{\expandafter\@slowromancap\romannumeral #1@}
\makeatother

\begin{document}

\title{Metric mean dimension of factor maps}

\author{Rui Yang
}
\address
{School of Mathematics, Northwest University, Xi'an, Shaanxi, 710127, P.R. China}

\email{zkyangrui2015@163.com}

\renewcommand{\thefootnote}{}


 
\renewcommand{\thefootnote}{\arabic{footnote}}
\setcounter{footnote}{0}

\begin{abstract}

Metric mean dimension is a metric-depedent quantity to characterize the topological complexity of systems with infinite topological entropy. In this paper, we investigate  metric mean dimension of factor maps.

(1) We introduce three  types of weighted metric mean dimensions  to characterize factor maps with infinite weighted topological entropy, and compare them with the metric mean dimensions of the factor system and the extension system. Furthermore, we establish  variational principles for weighted metric mean dimension.

(2) We introduce  relative  conditional metric mean dimension for factor maps with infinite relative topological conditional entropy, and prove that it coincides with  relative metric mean dimension.

(3) In the context of random  dynamical systems, the natural projection from the skew product to its driving system is a one-Lipschitz map.  We introduce  random average metric mean dimension  and use it to establish a topological Abramov-Rokhlin formula for the certain one-Lipschitz map. As an application, we obtain an inequality that links the metric mean dimensions of the driving system, the skew product system, and the inherent non-autonomous dynamical systems  from  the random dynamical system.

\end{abstract}


\maketitle
\pagestyle{plain}  
\section{Introduction}

  By a pair $(X,T)$ we mean a  topological dynamical system (TDS for short), where $X$ is a compact  metrizable topological  space and $T:X\rightarrow X$ is a homeomorphism map.   We sometimes use  a triplet $(X,d,T)$ to emphasize the compatible metric $d$ on $X$. Let $M(X,T)$, $E(X,T)$ denote the sets of $T$-invariant, $T$-ergodic  Borel probability measures on $X$, respectively.
  
  Topological entropy, introduced by Adler et al. \cite{akm65}, is a  vital topological invariant for describing the topological complexity of dynamical systems. It is a powerful tool in revealing chaotic phenomena of positive entropy systems and describing the "size" of fractal-like sets in dimension theory. However, we obtain no additional information about the dynamical behaviors of systems with infinite topological entropy. Inspired by the definition of box dimension in  fractal geometry, metric mean dimension was introduced by Lindenstrauss and Weiss \cite{lw00} to capture the dynamics of infinite entropy systems. 
  
  A map $\pi: X\rightarrow Y$ is called a \emph{factor map} between two dynamical systems $(X, T)$ and $(Y,S)$ if $\pi$ is a continuous surjective map satisfying $S\circ \pi(x)=\pi\circ T(x)$ for all $x\in X$; in this case,  one says that $(X,T)$ is  an $\emph{extension system}$ of $(Y,S)$, or $(Y,S)$ is a \emph{factor system} of $(X,T)$.  Additionally, if $\pi$ is a homeomorphism  map, $\pi$ is called a  \emph{conjugation map} between  $(X, T)$ and $(Y,S)$.

It is well-known that for a factor map between two dynamical systems, the topological entropy of the extension system is no less than the one of the factor system.  As for metric mean dimension,  for some certain factor maps between two dynamical systems,  the metric mean dimension of the factor system is  strictly greater than that of the extension system. Except this  unusual fact,   for factor maps few results are known regarding the metric mean dimensions of the factor system and extension system. For this reason, \emph{the goal of this paper is to investigate the metric mean dimensions of factor maps}. We mainly focus on three  entropy-like quantities that have been  used to characterize the complexity of factor maps: weighted topological entropy, relative topological conditional entropy, and random topological entropy.


Weighted topological entropy of factor maps has its roots in fractal geometry and concerns geometric objects with self-similarity. More precisely, let \( T \) be an endomorphism on the 2-dimensional torus \( \mathbb{T}^2 \), where \( T(u) = Au \pmod{1} \) for every \( u \in \mathbb{T}^2 \), and \( A = \text{diag}(m_1, m_2) \) with \( 2 \leq m_1 < m_2 \). In two seminal work, Bedford \cite{bed84} and McMullen \cite{mc84} independently calculated the Hausdorff dimension of a class of \( T \)-invariant subsets of \( \mathbb{T}^2 \) called \emph{self-affine Sierpiński gaskets}. Later, for any compact \( T \)-invariant subset \( K \subset \mathbb{T}^2 \), Kenyon and Peres \cite{kp96} proved that the Hausdorff dimension of \( K \) is given by the following variational principle:
\begin{align}\label{equ 1.1}
\dim_H(K) = \sup_{\substack{\mu \in {M}(\mathbb{T}^2, T) \\ \mu(K) = 1}} \left\{ \frac{1}{\log m_2} \, h_{\mu}(T) + \left( \frac{1}{\log m_1} - \frac{1}{\log m_2} \right) h_{\pi_{*}\mu}(S) \right\},
\end{align}
where \( \pi: (\mathbb{T}^2, T) \to (\mathbb{T}, S) \) is the factor map given by \( \pi:(x,y) \mapsto x \), and \( S(x) = m_1x \pmod{1} \). For any factor map \( \pi: (X,T) \to (Y,S) \) and a weight \( \mathbf{a} = (a_1, a_2) \) with \( a_1 > 0 \) and \( a_2 \geq 0 \), using Carathéodory-Pesin structures \cite{p97} Feng and Huang \cite{fw1616} introduced the weighted Bowen topological entropy, whose definition resembles that of Hausdorff dimension. Furthermore, they proved a variational principle for it \cite[Theorem 1.4]{fw1616} analogous to the equality \eqref{equ 1.1}:
\begin{align}\label{equ 1.2}
h_{{top}}^{\mathbf{a}}(T) = \sup_{\mu \in {M}(X,T)} \left\{ a_1h_{\mu}(T) + a_2h_{\pi_{*}\mu}(S) \right\},
\end{align}
where \( h_{{top}}^{\mathbf{a}}(T) \) denotes the weighted Bowen topological entropy of \( X \), \( h_{\mu}(T) \) denotes the measure-theoretic entropy of $\mu$, and \( \pi_{*}\mu = \mu \circ \pi^{-1} \) is the push-forward of \( \mu \) under \( \pi \).  

From the perspective of local entropy theory, for a family of open covers on the phase spaces of factor maps, Zhu \cite{z24} introduced the local weighted topological entropy in both topological and measure-theoretic settings, and related them via a local variational principle. Later, Tsukamoto \cite{t23} presented a new approach to weighted topological entropy, whose definition is analogous to the classical topological entropy defined by spanning sets. A variational principle for this new weighted topological entropy was obtained, which in turn proves its equivalence with Feng and Huang's definition. For known applications of these weighted topological entropies, one can refer to \cite{fw1616, t23, z24, t25}. Building on these work, we introduce the corresponding weighted metric dimensions of factor maps and pose the following question:
  
\emph{Question 1:
\begin{itemize}
\item [(1)] Can we establish variational principles for the weighted metric mean dimensions?
\item [(2)] What are the precise relations between the weighted metric mean dimensions and the metric mean dimensions of the factor systems and extension systems?
\item [(3)] Under which conditions, do these weighted metric mean dimensions coincide?
\end{itemize}}

Our results give a partial answer to Question 1.  

\begin{thm}\label{thm 1.1}
 Let  $\pi:(X,d_X, T)\rightarrow (Y,d_Y,S)$ be a factor map, and let \rm {\textbf{a}}$=(a_1,a_1)$  with $a_1> 0$ and $a_2\geq 0$.  If $\underline{\rm mdim}_M(S,Y,d_Y)=\overline{\rm mdim}_M(S,Y,d_Y)$, then 
\begin{align*}
\overline{\rm mdim}_M^{\textbf{a}}(T,X,\{d_X,d_Y\})=\max_{\mu \in M(X,T)}\{a_1 \overline{\rm {mdim}}_{\mu}(T,d_X) +a_2 \overline{\rm {mdim}}_{\pi_{*}\mu}(S,d_Y)\},
\end{align*}
where  $\overline{\rm mdim}_M^{\textbf{a}}(T,X,\{d_X,d_Y\})$ denotes the weighted  upper metric mean dimension of $X$ with respect to $\pi$, and  $\overline{\rm {mdim}}_{\mu}(T,d_X)$ denotes the measure-theoretic upper  metric mean dimension of $\mu$.
\end{thm}  
  
\begin{thm}\label{thm 1.2}
Let  $\pi:(X,d_X, T)\rightarrow (Y,d_Y,S)$ be a factor map.  If $\overline{\rm mdim}_M({S},{Y},d_Y)=0$, then for any $\omega \in (0,1]$,
\begin{align*}
\overline{\rm mdim}_M^{(\boldsymbol{\omega, 1-\omega})}({T},{X},\{d_X,d_Y\})
=\overline{\rm {mdim}}_M^{(\boldsymbol{\omega, 1-\omega}),B}(T,X,\{d_X,d_Y\})
=\omega \cdot  \overline{\rm mdim}_M({T},{X},d_X),
\end{align*}
where $\overline{\rm mdim}_M({T},{X},d_X)$ denotes the upper metric mean dimension of $X$.
\end{thm}

Two other entropy-like topological invariants used to characterize the complexity of factor maps are the \emph{relative topological entropy} \cite{lw97, ds02} and the \emph{relative topological conditional entropy} \cite{zz23}. Both of them  are closely related to the existence of measures of maximal relative entropy. For instance, the authors \cite{zz23}  proved that any factor map with zero relative topological conditional entropy admits invariant measures with maximal relative entropy. The relative topological conditional entropy is a fusion of the relative topological entropy \cite{lw97} and the topological conditional entropy \cite{mis76}. For factor maps with infinite relative topological entropy, Wu \cite{w23} introduced the relative metric mean dimension and established variational principles for it in terms of several types of relative measure-theoretic $\epsilon$-entropies of invariant measures. We introduce the relative  conditional metric mean dimension. This naturally raises the following question:

\emph{Question 2:  For a factor map,  do  the relative metric mean dimension and the relative topological  conditional  metric mean dimension coincide?}

\begin{thm}\label{thm 1.3}
Let $\pi:(X,d_X, T)\rightarrow (Y,S)$ be a factor map. Then 
$$\overline{\rm mdim}_M^{*}(X,T|\pi,d_X)=\overline{\rm mdim}_M(X,T|\pi,d_X),$$
where  $\overline{\rm mdim}_M^{*}(X,T|\pi,d_X)$ denotes the   relative  conditional  metric mean dimension of  $\pi$, and $\overline{\rm mdim}_M(X,T|\pi,d_X)$ denotes the relative   metric mean dimension of $\pi$.
\end{thm}

It is well-known \cite[Theorem 17]{bow71} that for a factor map between two dynamical systems, the topological entropy of the extension system is bounded above by the sum of the topological entropy of the factor system and the relative topological entropy. We provide an example of a Lipschitz factor map between two dynamical systems where the analogous Bowen's factor formula for metric mean dimension fails. This means one cannot expect the classical Bowen factor formula is valid for metric mean dimension. Hence, we turn to a certain one-Lipschitz factor map generated by a random dynamical system. One may think of random dynamical systems as topological dynamical systems with some disturbances on  phase spaces. More precisely, let \( (\Omega, \mathcal{F}, \mathbb{P}, \theta) \) be a measure-preserving system. A continuous random dynamical system \( X \) over \( (\Omega, \mathcal{F}, \mathbb{P}, \theta) \) is a map \( T: \mathbb{Z} \times \Omega \times X \rightarrow X \) that is continuous in \( x \), measurable in \( \omega \), and satisfies the cocycle property: for all \( \omega \in \Omega \), \( m, n \in \mathbb{Z} \), \( T_{\omega}^{n+m}(x) = T_{\theta^n \omega}^{m} \circ T_{\omega}^{n}(x) \), where \( T_{\omega} = T(1, \omega, \cdot) \) and the iteration is given by
\[
T_{\omega}^n := 
\begin{cases} 
T_{\sigma^{n-1}\omega} \circ \cdots \circ T_{\sigma\omega} \circ T_{\omega}, & \text{for } n \geq 1 \\
\text{id}, & \text{for } n = 0 \\
(T_{\sigma^{n}\omega})^{-1} \circ \cdots \circ (T_{\sigma^{-2}\omega})^{-1} \circ (T_{\sigma^{-1}\omega})^{-1}, & \text{for } n \leq -1.
\end{cases}
\]
The cocycle property associates a dynamical system on \( \Omega \times X \) via the skew product transformation \( \Theta(\omega, x) = (\theta\omega, T_{\omega}x) \). Moreover, varying \( \omega \in \Omega \) yields a family of non-autonomous dynamical systems \( \{(X, \{T_\omega^n\}_n)\}_{\omega \in \Omega} \). Roughly speaking, the random topological entropy of \( X \) is the average of the topological entropies of these non-autonomous systems with respect to \( \mathbb{P} \). Notice that the natural projection \( \pi_{\Omega}: (\Omega \times X, d_{\Omega} \times d_X) \rightarrow (\Omega, d_{\Omega}) \) onto its first coordinate is a Lipschitz factor map. If no disturbances are imposed on \( (X, d_X, T) \), i.e., \( T_{n,\omega}x = T^n x \), the product formula of metric mean dimension   \cite[Proposition 3.12]{yz25} shows that
\begin{align}\label{equ 4.9}
\overline{\text{mdim}}_M(\Theta, \Omega \times X, d_{\Omega} \times d_X) = \text{mdim}_M(\theta, \Omega, d_{\Omega}) + \overline{\text{mdim}}_M(T, X, d_{X}),
\end{align}
where \( d_{\Omega} \times d_X \) is the product metric on \( \Omega \times X \). Then it is natural to ask the following question:

\emph{Question 3: Let $\{T_{\omega}\}_{\omega}$ be a family of continuous transformations on $X$. Is there an analogous equality to (\ref{equ 4.9}) that links the metric mean dimensions of $\Omega \times X$ and $\Omega$?}

Inspired by the work \cite{b99, rgy25}, we introduce the random average metric mean dimension for random dynamical systems and establish a topological Abramov-Rokhlin formula, which provides a positive answer to Question 3.

\begin{thm} \label{thm 1.4}
Let $(\Omega, d_{\Omega}, \theta)$ be a TDS such that ${\rm mdim}_M(\theta, \Omega, d_{\Omega}) < \infty$. Let $T = (T_{n,\omega})$ be a random dynamical system over the measure-preserving system $(\Omega, \mathcal{B}(\Omega), \mathbb{P}, \theta)$. Then 
\begin{align*}
\overline{\rm mdim}_M(\Theta, \Omega \times X, d_{\Omega} \times d_X) = {\rm mdim}_M(\theta, \Omega, d_{\Omega}) + \overline{\rm Amdim}_M(T, X, \{d_{\Omega}, d_X\}),
\end{align*}
where $\overline{\rm Amdim}_M(T, X, \{d_{\Omega}, d_X\})$ denotes the random average upper metric mean dimension of $X$.
\end{thm}

As an application of the above Abramov-Rokhlin formula, we establish a quantitative relation between the metric mean dimensions of the skew product, the driving system, and these non-autonomous dynamical systems (i.e., Corollary \ref{cor 5.8}).

The rest of this paper is organized as follows. In Section \ref{sec 2}, we recall the definitions of metric mean dimension. In Section \ref{sec 3}, we introduce the weighted metric mean dimension and prove Theorems \ref{thm 1.1} and \ref{thm 1.2}. In Section \ref{sec 4}, we introduce the relative  conditional metric mean dimension and prove Theorem \ref{thm 1.3}. In Section \ref{sec 5}, we introduce the random average metric mean dimension and prove Theorem \ref{thm 1.4}.

\section{Preliminary}\label{sec 2}
In this section, we recall the precise definitions of metric mean dimension.

Let $(X,d,T)$ be a TDS. One says that a set  $E\subset X$ is a $(d,\epsilon)$-spanning set of $X$ if for any $x\in E$, there exists $y \in E$ such that $d(x,y)\leq \epsilon $; a  set $F\subset X$ is a $(d,\epsilon)$-separated set of $X$ if for any distinct $x,y \in F$, one has $d(x,y)>\epsilon$.  Denote by $r(X,d,\epsilon)$ and $s(X,d,\epsilon)$ the smallest cardinality of $(d,\epsilon)$-spanning sets  of $X$ and the  largest cardinality  of  $(d,\epsilon)$-separated sets of $X$, respectively.

Given $n\in \mathbb{N}$ and $x,y \in X$, we define a $n$-th Bowen metric on $X$ as
$$d_n(x,y)=\max_{0\leq j \leq n-1}d(T^jx,T^jy).$$
Then the Bowen ball of $x$ with radius $\epsilon >0$ in the Bowen metric $d_n$ is given by
$$B_n(x,\epsilon):=\{y\in X:d_n(x,y)<\epsilon\}.$$
The classical topological entropy of $X$ is defined by
\begin{align*}
h_{top}(T,X)=&\lim_{\epsilon \to 0}\limsup_{n \to \infty}\frac{ \log r(X,d_n,\epsilon)}{n}\\
=&\lim_{\epsilon \to 0}\limsup_{n \to \infty}\frac{ \log s(X,d_n,\epsilon)}{n}.
\end{align*}

To  classify TDSs with infinite topological entropy,  inspired by  the definition of  the Box dimension in fractal geometry, Lindenstrauss and Weiss \cite{lw00} introduced the metric mean dimension of $X$. More precisely, the upper metric mean dimension of $X$ is defined by
\begin{align*}
\overline{{\rm mdim}}_M(T,X,d)=&\lim_{\epsilon \to 0}\frac{1}{\logf}\limsup_{n \to \infty}\frac{ \log r(X,d_n,\epsilon)}{n}\\
=&\lim_{\epsilon \to 0}\frac{1}{\logf}\limsup_{n \to \infty}\frac{ \log s(X,d_n,\epsilon)}{n}.
\end{align*}

Replacing $\limsup_{\epsilon \to 0}$ by $\liminf_{\epsilon \to 0}$, one defines the corresponding lower metric mean dimension  $\underline{{\rm mdim}}_M(T,X,d)$ of $X$. We denote by ${{\rm mdim}}_M(T,X,d)$ the metric mean dimension of $X$ if the upper and lower metric mean dimensions of $X$ coincide. Clearly, every TDS with finite topological entropy has zero metric mean dimension. Besides, unlike the classical topological entropy, metric mean dimension depends on the compatible metrics on $X$. Therefore, it is a  metric-depedent quantity to capture the dynamics of infinite entropy systems. 

 An equivalent definition  of metric mean dimension is  given using the language of open covers \cite[Proposition 3.3]{yz25}. Let $\mathcal{C}_X^o$ denote the set of all finite open covers of $X$.  Given  $\UU\in \mathcal{C}_X^o$, the \emph{diameter} of $\UU$, denoted by $\diam (\UU,d)$, is the  maximum  of the diameters of the elements in $\UU$; the Lebesgue number of $\UU$, denoted by $\Leb(\UU,d)$, is the largest  positive real number $\epsilon$ such that each $d$-open ball $B_d(x,\epsilon)$ is contained in some element of $\UU$.   We drop  the metric $d$  for the two notations $\diam (\UU,d)$ and $\Leb(\UU,d)$  if  the metric on  underlying space is clear. 

 Given  a non-empty subset $Z\subset X$, we define $N(\UU,Z)$ as the minimal cardinality of the subfamily of $\UU$ needed to cover the set $Z$. We write $N(\UU):= N(\UU,X)$ if $Z=X$.  Denote by
$$\UU^n=\vee_{j=0}^{n-1} T^{-j}\UU:=\{\cap_{j=0}^{n-1}T^{-j}U_{i_j}:U_{i_j} \in \UU,~ 0\leq j \leq n-1\}$$
the $n$-th join of $\UU$. The \emph{topological entropy of $\UU$} is defined by 
$$h_{top}(T,\UU)=\lim_{n \to \infty}\frac{\log N(\UU^n)}{n}.$$

\begin{prop} \label{prop 2.1}
Let $(X,d,T)$ be a TDS.  Then
\begin{align*}
\overline{{\rm mdim}}_M(T,X,d)=&\limsup_{\epsilon \to 0}\frac{1}{\logf}\inf_{\diam (\UU,d)\leq \epsilon}h_{top}(T,\UU)\\
\underline{{\rm mdim}}_M(T,X,d)=&\liminf_{\epsilon \to 0}\frac{1}{\logf}\inf_{\diam (\UU,d)\leq \epsilon}h_{top}(T,\UU),
\end{align*}
where the infimum ranges over the set of  $\UU\in \mathcal{C}_X^o$ with $\diam (\UU)\leq \epsilon$.
\end{prop}

One says that \cite{yz25} $\pi:(X,d_X,T) \rightarrow (Y,d_Y,S)$ is a \emph{Lipschitz factor map} if $\pi$ is a Lipschitz and surjective  map satisfying $S\circ \pi(x)=\pi\circ T(x)$ for all $x\in X$. Additionally,  it is a bi-Lipschitz conjugate map if $\pi$ is a bijection and bi-Lipschitz map satisfying $S\circ \pi(x)=\pi\circ T(x)$ for all $x\in X$.

\begin{prop} \cite[Proposition 3.7]{yz25}
Let $\pi:(X,d_X,T) \rightarrow (Y,d_Y,S)$ be a  Lipschitz factor map. Then 
\begin{align*}
\overline{{\rm mdim}}_M(T,X,d_X) \geq& \overline{{\rm mdim}}_M(S,Y,d_Y),\\
\underline{{\rm mdim}}_M(T,X,d_X) \geq& \underline{{\rm mdim}}_M(S,Y,d_Y).
\end{align*}
Additionally, the equalities hold if $\pi$ a  bi-Lipschitz conjugate map. 
\end{prop} 

For any a factor map $\pi:(X,T) \rightarrow (Y,S)$, Bowen  \cite[Theorem 17]{bow71} proved that
$$h_{top}(Y,S)\leq h_{top}(X,T)\leq h_{top}(Y,S) +\sup_{y\in Y}h_{top}(T,\pi^{-1}y).$$

We give an example to verify that the analogous Bowen factor formula fails for metric mean dimension. Recall that   upper and lower box dimensions  of $X$ are  respectively given by 
\begin{align*}
\overline{\rm dim}_B(X,d)=\limsup_{\epsilon \to 0}\frac{\log N_{\epsilon}(X,d)}{\logf},~
\underline{\rm dim}_B(X,d)=\liminf_{\epsilon \to 0}\frac{\log N_{\epsilon}(X,d)}{\logf},
\end{align*}
where $N_{\epsilon}(X,d)$ denotes the smallest cardinality of $d$-open balls $B_d(x,\epsilon)$ needed to cover $X$. Here, the box dimension becomes unchanged if $N_{\epsilon}(X,d)$ is replaced by $r(X,d,\epsilon)$ or $s(X,d,\epsilon)$.

Now consider the full shift  $(X^{\mathbb{Z}},d^{\mathbb{Z}},\sigma)$ over $X$, where $\sigma: X^{\mathbb{Z}} \rightarrow X^{\mathbb{Z}}$ is the left shift map given by $\sigma((x_n)_{n \in \mathbb{Z}})=(x_{n+1})_{n \in \mathbb{Z}}$,  and $$d^{\mathbb{Z}}(x,y)=\sum_{n\in \mathbb{Z}} \frac{d(x_n,y_n)}{2^{|n|}}$$ is a   product metric compatible with the product topology of $X^{\mathbb{Z}}$. The  authors in \cite[Theorem 5]{vv17} showed that
\begin{align}\label{equ 2.1}
\overline{{\rm mdim}}_M(\sigma,X^{\mathbb{Z}},d^{\mathbb{Z}}) =\overline{\rm dim}_B(X,d),
~\underline{{\rm mdim}}_M(\sigma,X^{\mathbb{Z}},d^{\mathbb{Z}}) =\underline{\rm dim}_B(X,d).
\end{align}

\begin{ex}\label{ex 2,3}
Let $A$ be $m (m\geq 2)$ symbols with the discrete topology. Fix $s\in (0,1)$.  Let $(A^{\mathbb{N}},d_s,\sigma)$ be the (one-sided) full shift over $A$, where
$$d_s(x,y)=(m^{-\frac{1}{s}})^{n(x,y)},$$
where $n(x,y)=\min \{k \in \mathbb{N}:x_k\not= y_k\}$. Then  $\overline{\rm dim}_B(A^{\mathbb{N}},d_s)=\underline{\rm dim}_B(A^{\mathbb{N}},d_s)=s$.   Hence, let $X=A^{\mathbb{N}}$ and $d_s^{\mathbb{Z}}$ as in (\ref{equ 2.1}). We  have  ${{\rm mdim}}_M(\sigma,X^{\mathbb{Z}},d_s^{\mathbb{Z}})=s$.

Choose $s_1,s_2 \in (0,1)$ with $s_1<s_2$, and  let $id: (\sigma,X^{\mathbb{Z}},d_{s_2}^{\mathbb{Z}}) \rightarrow (\sigma,X^{\mathbb{Z}},d_{s_1}^{\mathbb{Z}}) $ be the indentity map. Then $id$ is a one-Lipschitz factor map. However, it does  not hold  that 
$${{\rm mdim}}_M(\sigma,X^{\mathbb{Z}},d_{s_2}^{\mathbb{Z}})\leq {{\rm mdim}}_M(\sigma,X^{\mathbb{Z}},d_{s_1}^{\mathbb{Z}})+ \sup_{y\in Y}\overline{{\rm mdim}}_M(\sigma, id^{-1}(y),d_{s_2}^{\mathbb{Z}}).$$
\end{ex}

\begin{proof}
Fix a sufficiently small $\epsilon >0$.  Choose a positive integer $N_0$ such that $$(m^{-\frac{1}{s}})^{N_0} <\epsilon \leq (m^{-\frac{1}{s}})^{N_0-1}.$$ Then for every $x\in A^{\mathbb{N}}$,  we have $B_{d_s}(x,\epsilon)=[x|_{N_0}]:=\{ y\in A^{\mathbb{N}}: y_n=x_n ~\forall n=0,...,N_0\}$. Using
this fact, we deduce that 
$$\overline{\rm dim}_B(A^{\mathbb{N}},d_s)=\underline{\rm dim}_B(A^{\mathbb{\textsc{N}}},d_s)= \lim_{\epsilon \to 0} \frac{\log m^{N_0+1}}{\log (m^{\frac{1}{s}})^{N_0}} \frac{\log( m^{\frac{1}{s}})^{N_0}}{\logf}=s.$$
The remaining statements are due to (\ref{equ 2.1}).
\end{proof}

\section{Weighted metric mean dimension of factor maps}\label{sec 3}

In this section, we  introduce three types of weighted metric mean dimensions of factor maps and  prove Theorems  \ref{thm 1.1} and  \ref{thm 1.2}.

\subsection{Zhu's local weighted topological entropy}
We first recall the local weighted  entropy of open covers in both topological and measure-theoretic settings \cite{z24}.

Let $\pi:(X,T) \rightarrow (Y,S)$ be a factor map and $\textbf{a}=(a_1,a_2)$ with $a_1> 0$ and $a_2\geq 0$. Given $\UU \in \CC_X^o$ and $\VV\in \CC_Y^o$, the  \emph{local \textbf{a}-weighted topological entropy  of  $(\UU,\mathcal{V})$} is defined by
$$h_{top}^{\textbf{a}}(T, \UU,\mathcal{V})=a_1 h_{top}(T,\UU)+a_2h_{top}(T,\pi^{-1}(\mathcal{V})),$$
and  the \emph{\textbf{a}-weighted topological entropy  of  $X$}  is defined by
$$h_{top}^{\textbf{a}}(T)=\sup_{\mathcal{V}\in \mathcal{C}_Y^0}\sup_{\UU \in \mathcal{C}_X^o}h_{top}^{\textbf{a}}(T, \UU,\mathcal{V}).$$
Since $\pi$ is surjective and continuous, one has  $h_{top}^{\textbf{a}}(T, \UU,\mathcal{V})=a_1 h_{top}(T,\UU)+a_2h_{top}(S,\mathcal{V}),$
and hence $h_{top}^{\textbf{a}}(T)=a_1h_{top}(T,X)+a_2h_{top}(S,Y)$.

Given $\mu \in M(X,T)$, recall that the local  entropy of $\UU$ with respect to (w.r.t.) $\mu$ \cite{rom03} is given by $$h_{\mu}(T,\UU)=\inf_{\alpha \succ \UU, \alpha \in  \mathcal{P}_X}h_{\mu}(T,\alpha),$$  where the infimum ranges over  all finite Borel  partitions of $X$ that are finer than $\UU$,  and $h_{\mu}(T,\alpha)$ is the usual measure-theoretic entropy of $\mu$ w.r.t. $\alpha$.

 The \emph{local \textbf{a}-weighted  measure-theoretic $\mu$-entropy  of  $(\UU,\mathcal{V})$} is defined by
$$h_{\mu}^{\textbf{a}}(T, \UU,\mathcal{V})=a_1 h_{\mu}(T,\UU)+a_2h_{\mu}(T,\pi^{-1}(\mathcal{V})),$$
and the \emph{\textbf{a}-weighted measure-theoretic entropy  of  $\mu$}  is defined by
$$h_{\mu}^{\textbf{a}}(T)=\sup_{\mathcal{V}\in \mathcal{C}_Y^0}\sup_{\UU \in \mathcal{C}_X^o}h_{\mu}^{\textbf{a}}(T, \UU,\mathcal{V}).$$
Using  \cite[Proposition 6]{rom03}, one has $h_{\mu}(T,\pi^{-1}(\mathcal{V}))=h_{\pi_{*} \mu}(S,\mathcal{V})$   for all $\mu \in M(X,T)$. Hence,  we obtain  that $h_{\mu}^{\textbf{a}}(T, \UU,\mathcal{V})=a_1 h_{\mu}(T,\UU)+a_2h_{\pi_{*}\mu}(S,\mathcal{V}),$  and $h_{\mu}^{\textbf{a}}(T)=a_1h_{\mu}(T)+a_2h_{\pi_{*}\mu}(S)$  since $h_{\mu}(T)=\sup_{\UU\in \CC_X^0}h_{\mu}(T,\UU)$ (cf. \cite[Theorem 3.5]{hyz11}). 

The local \textbf{a}-weighted topological entropy  and the  local \textbf{a}-weighted measure-theoretic $\mu$-entropy   are related by the following local variational principle \cite[Theorem 2.9]{z24}.

\begin{lem}\label{lem 4.1}
Let $\pi:(X,T) \rightarrow (Y,S)$ be a factor map and \rm {\textbf{a}}$=(a_1,a_2)$  with $a_1> 0$ and $a_2\geq 0$.  Let $\UU \in \CC_X^o$ and $\VV\in \CC_Y^o$. Then
\begin{align*}
h_{top}^{\textbf{a}}(T, \UU,\mathcal{V})=\sup_{\mu \in M(X,T)} h_{\mu}^{\textbf{a}}(T, \UU,\mathcal{V})
=\sup_{\mu \in E(X,T)} h_{\mu}^{\textbf{a}}(T, \UU,\mathcal{V}).
\end{align*}
\end{lem}
\begin{proof}
The invariant case is due to Zhu \cite{z24}.   For the ergodic case, it   can be proved  using  the ergodic decomposition theorem of $h_{\mu}(T,\UU)$ \cite[Theorem 3.13]{hyz11}.
\end{proof}


Now we define  weighted metric mean dimension of $X$ with respect to a factor map.

Let $\pi:(X,d_X,T) \rightarrow (Y,d_Y,S)$ be a factor map and $\textbf{a}=(a_1,a_2)$ with $a_1> 0$ and $a_2\geq 0$. Let 
$$h_{top}^{\textbf{a}}(T,X,\{d_X,d_Y\},\epsilon):= a_1\inf_{\diam (\UU,d_X)\leq \epsilon}h_{top}(T,\UU) +a_2 \inf_{\diam (\VV,d_Y)\leq \epsilon}h_{top}(S,\VV)$$
denote  the \textbf{a}-weighted $\epsilon$-topological entropy of $X$.

\begin{df}\label{df 5.3}
The  \textbf{a}-weighted  upper and lower metric mean dimensions  of $X$ are respectively defined by
\begin{align*}
\overline{\rm mdim}_M^{\textbf{a}}(T,X,\{d_X,d_Y\})&=\limsup_{\epsilon \to 0}\frac{1}{\logf}h_{top}^{\textbf{a}}(T,X,\{d_X,d_Y\},\epsilon),\\ 
\underline{\rm mdim}_M^{\textbf{a}}(T,X,\{d_X,d_Y\})&=\liminf_{\epsilon \to 0}\frac{1}{\logf}h_{top}^{\textbf{a}}(T,X,\{d_X,d_Y\},\epsilon).
\end{align*}
\end{df}

\begin{rem}
\begin{itemize}
\item [(1)] The weighted metric mean dimension of $X$ not only depends on the  compatible metrics on $X$ and $Y$, but also depends on the  weight $\textbf{a}=(a_1,a_2)$. 
\item [(2)]  The weight $(1,0)$ reclaims the metric mean dimension of $X$ given  by Lindenstrauss and Weiss \cite{lw00}. Therefore, the weighted metric mean dimension is an extension of the metric mean dimension of a single dynamical system to a factor map between two dynamical systems.
\end{itemize}
\end{rem}

Next we recall the definition of measure-theoretic metric mean dimension of invariant measures \cite[Definition 3.24]{yz25}.

\begin{df}
$(a)$ Let $\mathrm{co}(E(X,T))$ be  the convex hull  of $E(X,T)$. For $\mu \in  \mathrm{co}(E(X,G))$, we can rewrite $\mu=\sum_{j=1}^k\lambda_j \mu_j$, where $ \sum_{j=1}^{k}\lambda_j=1, \mu_j \in E(X,T), ~0\leq \lambda_j\leq 1,~j=1,...,k$. The  measure-theoretic $\epsilon$-entropy of $\mu$ is defined as
\begin{align*}
F(\mu, \epsilon):=\sum_{j=1}^{k}\lambda_j \inf_{\diam (\UU,d)\leq \epsilon } h_{\mu}(T,\UU).
\end{align*}

$(b)$ For $\mu \in  M(X,G)$, let $M_{T}(\mu)$  denote the collection of all families $\{\mu_{\epsilon}\}_{\epsilon >0} \subset \mathrm{co}(E(X,T))$  that converge to $\mu$  as $\epsilon \to 0$ in the weak$^{*}$-topology. We define  the  measure-theoretic upper and lower  metric mean dimensions of $\mu$ as
\begin{align*}
\begin{split}
\overline{\rm {mdim}}_{\mu}(T,d)&=\sup_{(\mu_{\epsilon})_{\epsilon}\in M_{T}(\mu) }\{\limsup_{\epsilon \to 0}\frac{F(\mu_{\epsilon}, \epsilon)}{\log \frac{1}{\epsilon}}\},\\
\underline{\rm {mdim}}_{\mu}(T,d)&=\sup_{(\mu_{\epsilon})_{\epsilon}\in M_{T}(\mu) }\{\liminf_{\epsilon \to 0}\frac{F(\mu_{\epsilon}, \epsilon)}{\log \frac{1}{\epsilon}}\},
\end{split}
\end{align*}
respectively.
\end{df}

\begin{thm}\cite[Theorem 1.1]{yz25}\label{thm 3.6}
Let $(X, d, T)$ be a TDS.  Then 
\begin{align*}
{\overline{\rm  mdim}}_M(T,X,d)&=\max_{\mu \in M(X,T)}\overline{\rm {mdim}}_{\mu}(T,d).
\end{align*}
\end{thm}

We are ready to prove Theorem \ref{thm 1.1}.

\begin{proof}[Proof of Theorem \ref{thm 1.1}]
It is easy to see that for every $\epsilon >0$ and $\mu \in  M(X,T)$,
\begin{align}\label{equ 4.2}
\begin{split}
\inf_{\diam (\UU,d_X)\leq \epsilon}h_{top}(T,\UU)&\leq  \sup_{\diam (\UU,d_X)\leq \epsilon, \atop  \Leb (\UU,d_X) \geq \frac{\epsilon}{4}}h_{top}(T,\UU)\leq\inf_{\diam (\UU,d_X)\leq \frac{\epsilon}{8}}h_{top}(T,\UU),\\
\inf_{\diam (\UU,d_X)\leq \epsilon}h_{\mu}(T,\UU)&\leq  \sup_{\diam (\UU,d_X)\leq \epsilon, \atop  \Leb (\UU,d_X) \geq \frac{\epsilon}{4}}h_{\mu}(T,\UU)\leq\inf_{\diam (\UU,d_X)\leq \frac{\epsilon}{8}}h_{\mu}(T,\UU).
\end{split}
\end{align}
Therefore, by Lemma \ref{lem 4.1}  we have
\begin{align}\label{equ 4.3}
&  \sup_{\mu \in E(X,T)}\{a_1\sup_{\diam (\UU,d_X)\leq \epsilon, \atop  \Leb (\UU,d_X) \geq \frac{\epsilon}{4}}h_{\mu}(T,\UU)+a_2 \sup_{\diam (\VV,d_Y)\leq \epsilon, \atop  \Leb (\VV,d_Y) \geq \frac{\epsilon}{4}}h_{\pi_{*}\mu} (S,\VV)\}\nonumber\\
=&  \sup_{\mu \in E(X,T)}\{\sup_{\diam (\UU,d_X)\leq \epsilon, \atop  \Leb (\UU,d_X) \geq \frac{\epsilon}{4}} \sup_{\diam (\VV,d_Y)\leq \epsilon, \atop  \Leb (\VV,d_Y) \geq \frac{\epsilon}{4}}(a_1h_{\mu}(T,\UU)+a_2 h_{\pi_{*}\mu} (S,\VV))\} \nonumber \\
=&\sup_{\diam (\VV,d_Y)\leq \epsilon, \atop  \Leb (\VV,d_Y) \geq \frac{\epsilon}{4}}\sup_{\diam (\UU,d_X)\leq \epsilon, \atop  \Leb (\UU,d_X) \geq \frac{\epsilon}{4}}h_{top}^{\textbf{a}}(T, \UU,\mathcal{V})\nonumber
\end{align}
Combing this inequality with  (\ref{equ 4.2}), for every $\epsilon >0$ we have
\begin{align}
h_{top}^{\textbf{a}}(T,X,\{d_X,d_Y\},\epsilon)
\leq \sup_{\mu \in E(X,T)}\{\inf_{\diam (\UU,d_X)\leq \frac{\epsilon}{8}} \inf_{\diam (\VV,d_Y)\leq \frac{\epsilon}{8}}(a_1h_{\mu}(T,\UU)+a_2 h_{\pi_{*}\mu} (S,\VV))\},
\end{align}
and hence there exists $\mu_{\epsilon} \in E(X,T)$ such that
$$h_{top}^{\textbf{a}}(T,X,\{d_X,d_Y\},\epsilon)-\epsilon
<\inf_{\diam (\UU,d_X)\leq \frac{\epsilon}{8}} \inf_{\diam (\VV,d_Y)\leq \frac{\epsilon}{8}}(a_1h_{\mu_\epsilon}(T,\UU)+a_2 h_{\pi_{*}\mu_\epsilon} (S,\VV))$$

Let $\{\epsilon_n\}$ be a  strictly decreasing sequence that converges to $0$ such that 
 $$\overline{\rm mdim}_M^{\textbf{a}}(T,X,\{d_X,d_Y\})=\lim_{n \to \infty}\frac{h_{top}^{\textbf{a}}(T,X,\{d_X,d_Y\},\epsilon_n)}{\log \frac{1}{\epsilon_n}}$$
 and for some $\mu\in M(X,T)$,  $\mu_{\epsilon_n} \to \mu$ as $n \to \infty$. Then $\pi_{*}\mu_{\epsilon_n} \to \pi_{*}\mu\in M(Y,S)$ as $n \to \infty$ since $\pi_{*}$ is continuous. For every $0<\epsilon\leq \epsilon_1$, we define $\lambda_{\epsilon}=\mu_{\epsilon_n}$  and $\nu_\epsilon =\pi_{*}\mu_{\epsilon_n}$ if $\epsilon\in (\epsilon_{n+1},\epsilon_n]$; if $\epsilon >\epsilon_1$, we  define $\lambda_{\epsilon}=\mu$  and $\nu_\epsilon =\pi_{*}\mu$. Then $(\lambda_\epsilon)_{\epsilon} \in M_{T}(\mu)$ and $(\nu_\epsilon)_{\epsilon} \in M_{S}(\pi_{*}\mu)$. Hence, we have
\begin{align}\label{inequ 3.4}
&\overline{\rm mdim}_M^{\textbf{a}}(T,X,\{d_X,d_Y\})\\
\leq &\limsup_{n \to \infty}\frac{1}{\log \frac{1}{\epsilon_n}}\{a_1\inf_{\diam (\UU,d_X)\leq \frac{\epsilon_n}{8}}h_{\mu_{\epsilon_n}}(T,\UU) +a_2 \inf_{\diam (\VV,d_Y)\leq \frac{\epsilon_n}{8}}h_{\pi_{*}\mu_{\epsilon_n}}(S,\VV)\} \nonumber \\
\leq & a_1 \limsup_{n \to \infty}\frac{1}{\log \frac{1}{\epsilon_n}} \inf_{\diam (\UU,d_X)\leq \frac{\epsilon_n}{8}}h_{\mu_{\epsilon_n}}(T,\UU)+ 
a_2 \limsup_{n \to \infty}\frac{1}{\log \frac{1}{\epsilon_n}} \inf_{\diam (\VV,d_Y)\leq \frac{\epsilon_n}{8}}h_{\pi_{*}\mu_{\epsilon_n}}(S,\VV) \nonumber\\
\leq & a_1 \overline{\rm {mdim}}_{\mu}(T,d_X) +a_2 \overline{\rm {mdim}}_{\pi_{*}\mu}(S,d_Y) \nonumber,
\end{align}
where we used the fact that for any $l>0$,  $$\overline{\rm {mdim}}_{\mu}(T,d)=\sup_{(\mu_{\epsilon})_{\epsilon}\in M_{T}(\mu) }\{\limsup_{\epsilon \to 0}\frac{F(\mu_{\epsilon}, l\epsilon)}{\log \frac{1}{\epsilon}}\}$$ for the last inequality.

 On the other hand,  by  Theorem \ref{thm 3.6} we have
\begin{align}\label{inequ 3.5}
&\sup_{\mu \in M(X,T)}\{a_1 \overline{\rm {mdim}}_{\mu}(T,d_X) +a_2 \overline{\rm {mdim}}_{\pi_{*}\mu}(S,d_Y)\} \nonumber\\
 \leq&  a_1\overline{\rm mdim}_M(T,X,d_X)+a_2{\rm mdim}_M(S,Y,d_Y)\nonumber \\
=&\overline{\rm mdim}_M^{\textbf{a}}(T,X,\{d_X,d_Y\}),
\end{align}
where we used Proposition \ref{prop 2.1} and $\overline{\rm mdim}_M(S,Y,d_Y)=\underline{\rm mdim}_M(S,Y,d_Y)$ for the last equality.

By (\ref{inequ 3.4}) and (\ref{inequ 3.5}), we complete the proof.
\end{proof}

\begin{rem}
In Theorem \ref{thm 1.1}, it is unclear whether  the corresponding  variational principle holds for the  weighted lower metric mean dimension of $X$.
\end{rem}

\subsection{Tsukamoto's approach to weighted topological entropy}

Let  $\pi:(X,d_X,T)\rightarrow (Y,d_Y,S)$ be a factor map. Fix $\epsilon>0$, $n \in \mathbb{N}$, $\omega \in [0,1]$, and let $\Omega$ be a non-empty subset of $X$. We  define $\#(\Omega,n,\epsilon)$ as the  smallest positive integer $m$ such that  $\Omega$ is  covered by  a family of open subsets $\{U_1,\cdots, U_m\}$  of $X$ with  $\diam(U_j,(d_X)_{n})<\epsilon$ for all $1\leq j \leq m$. We put
\begin{align*}
\#^{\omega}(\pi,n,\epsilon)
=\inf\left\{\sum_{i=1}^k
(\#(\pi^{-1}(V_i),n,\epsilon))^{\omega}\right\},
\end{align*} 
where the infimum ranges over all open covers $\{V_1,,,.,V_k\}$ of $Y$ with  $\diam(V_i,(d_{Y})_{n})<\epsilon$ for all $1\leq i \leq k$.

The \emph{weighted topological entropy of $X$} is defined by 
$$h_{top}^{\omega}(X,T|\pi)=\lim_{\epsilon \to 0}\lim_{n \to \infty}\frac{1}{n}\log  \#^{\omega}(\pi,n,\epsilon),$$
where the $\lim_{n \to \infty}$ exists since  $\{\log  \#^{\omega}(\pi,n,\epsilon)\}$ is  a sub-additive sequence in $n$.

\begin{df}\label{df 5.6}
We define the $\omega$-weighted upper and lower metric mean dimensions  of $X$ with respect to $\pi$ as 
\begin{align*}
\overline{\rm mdim}_M^{\omega}(X,T|\pi,\{d_X,d_Y\})&=\limsup_{\epsilon \to 0}\frac{1}{\logf}\lim_{n \to \infty}\frac{1}{n}\log  \#^{\omega}(\pi,n,\epsilon),\\
\underline{\rm mdim}_M^{\omega}(X,T|\pi,\{d_X,d_Y\})&=\liminf_{\epsilon \to 0}\frac{1}{\logf}\lim_{n \to \infty}\frac{1}{n}\log  \#^{\omega}(\pi,n,\epsilon),
\end{align*}
respectively.
\end{df}

The following two examples are helpful  for understanding the $\omega$-weighted  metric mean dimensions of  factor maps.

\begin{ex}\label{ex 4.7}
$(1)$  Given  two TDSs $(X,d_X,T)$ and $(Y,d_Y,S)$, let $$\pi_X: (X\times Y, d_X\times d_Y,T\times S)\rightarrow (X,d_X,T)$$ be  the coordinate projection from $X\times Y$ to $X$, where $$d_X\times d_Y\big((x_1,y_1),(x_2,y_2)\big)=\max\{d_X(x_1,x_2), d_Y(y_1,y_2)\}.$$ Then $\pi_X: X\times Y \rightarrow Y$ is a one-factor map.   Assume that $\overline{\rm mdim}_M(T,X,d_X)=\underline{\rm mdim}_M(T,X,d_X)$. Then for every $\omega \in [0,1]$,
\begin{align}\label{equ 4.4}
\begin{split}
\overline{\rm mdim}_M^{\omega}(X\times Y,T\times S|\pi_X,\{d_X\times d_Y,d_Y\})
= {\rm mdim}_M(T,X,d_X) + \omega \cdot \overline{\rm mdim}_M(S,Y,d_Y). 
\end{split}
\end{align}

$(2)$ Let $(X,d,T)$ be a TDS. Put
$$\widetilde{X}:=\{(x_n)_{n\in \mathbb{Z}} \in X^{\mathbb{Z}}: x_{n+1}=T(x_n),x_n \in X, \forall n\in \mathbb{Z}\}.$$
Then $\widetilde{X}$ is a closed subset of $X^{\mathbb{
Z}}$, and  is metrizable by the metric $$\widetilde{d}(x,y )=\sum_{n\in \mathbb{Z}}\frac{d(x_n,y_n)}{2^{|n|}},$$
where $x=(x_n)_{n\in \mathbb{Z}},y= (y_n)_{n\in \mathbb{Z}} \in \widetilde{X}$.
Let $\widetilde{T}$ be the left shift on $\widetilde{X}$. Then the projection $$\pi:(x_n)_{n\in \mathbb{Z}} \in \widetilde{X} \mapsto x_0 \in X$$ to its $0$-th coordinate is a factor map between $(\widetilde{X},\widetilde{d},\widetilde{T})$ and $(X,d,T)$. This is also known as  the {natural extension} of $(X,T)$. Then  for every $\omega \in [0,1]$,
\begin{align}\label{equ 4.5}
\overline{\rm mdim}_M^{\omega}(\widetilde{X},\widetilde{T}|\pi,\{\widetilde{d},d\})= \overline{\rm mdim}_{M}(T,X,d).
\end{align}
\end{ex}

\begin{proof}
(1). Notice that the upper metric mean dimension of $Y$ is  equivalently given by
$$\overline{\rm mdim}_M(S,Y,d_Y)=\limsup_{\epsilon \to 0}\frac{1}{\logf}\lim_{n \to \infty}\frac{\log \#(Y,n,\epsilon)}{n}.$$

Fix $\epsilon >0$ and $n\in \mathbb{N}$. Observe that  $\pi_X^{-1}(A)= A\times Y$ for any $A \subset X$ and  $(d_X\times d_Y)_n=\max\{(d_X)_n,(d_Y)_n\}$. Now cover $X$  by the open sets $\{V_j\}_{j=1}^m$ of $X$ with $\diam(V_j, (d_X)_n)<\epsilon$ for every $1\leq j \leq m$. Then 
$\#(\pi_X^{-1}(V_j),n,\epsilon)= \#(Y,n,\epsilon)$ for every $1\leq j\leq m$. So $\#^{\omega}(\pi,n,\epsilon)= \#(X,n,\epsilon)(\#(Y,n,\epsilon))^{\omega}$. This yields that
\begin{align*}
\overline{\rm mdim}_M^{\omega}(X\times Y,T\times S|\pi_X,\{d_X\times d_Y,d_Y\})
= {\rm mdim}_M(T,X,d_X) + \omega \cdot \overline{\rm mdim}_M(S,Y,d_Y).
\end{align*}

(2).  Fix $\epsilon >0$  and choose $N_0$ sufficiently large such that $\sum_{n\geq N_0}\frac{\diam(X,d)}{2^{|n|}}<\frac{\epsilon}{4}.$  Let $E$ be  a $(n,\frac{\epsilon}{8})$-spanning set of $X$. Then  $X=\cup_{x\in E}B_n(x,\frac{\epsilon}{8})$.  
Assume that $\{B_d(y_j,\frac{\epsilon}{8})\}_{j=1}^M$ is a finite open cover of $X$  consisting of $d$-open balls, where $M$ is a positive integer depending on $\epsilon$.  We set
$$\widetilde{U}:=\{x\in \widetilde{X}: x_{j}\in B(y_{i_j},\frac{\epsilon}{8})~ \text{for}~j=-N_0+1,...,-1, n, n+1,...,n+N_0-1\}.$$
Varying the indexes $i_j$, the  total number of such open sets  is at most $M^{2N_0}$. Fix $x\in E$. If $\widetilde{U}\cap \pi^{-1}B_n(x,\frac{\epsilon}{8})\not=\emptyset$, 
then  $\widetilde{d}_n(y_1,y_2)< \frac{3}{4}\epsilon$ for any two  distinct $y_1,y_2 \in \widetilde{U}\cap \pi^{-1}B_n(x,\frac{\epsilon}{8})$. 

Indeed,  by definition we have $d((y_1)_j,(y_2)_j)<\frac{\epsilon}{4}$ for  $j=-N_0+1,...,n+N_0-1$. Then for any $0\leq k\leq n-1$,
$$\widetilde{d}(\widetilde{T}^k(y_1),\widetilde{T}^k(y_2))< \sum_{-N_0<j<N_0}\frac{d((\widetilde{T}^k(y_1))_j,(\widetilde{T}^k(y_2))_j)}{2^{|j|}} +  \frac{\epsilon}{4}<\frac{3}{4}\epsilon.$$
This implies that $\diam (\widetilde{U}\cap \pi^{-1}B_n(x,\frac{\epsilon}{4}),\widetilde{d}_n)<\epsilon$. Thus, we have 
$$\#^{\omega}(\pi,n,\epsilon)\leq  r(X,d_n,\frac{\epsilon}{8})\cdot M^{2 N_0 \omega}$$ for all $\omega \in [0,1]$. Taking the corresponding limits, we get
$$\overline{\rm mdim}_M^{\omega}(\widetilde{X},\widetilde{T}|\pi,\{\widetilde{d},d\})\leq  \overline{\rm mdim}_{M}(T,X,d).$$

On the other hand, it is clear that $\#^{\omega}(\pi,n,\epsilon)\geq  \#(X,n,\epsilon)$  for every $\omega\in [0,1]$. This shows the reverse inequality  $$\overline{\rm mdim}_M^{\omega}(\widetilde{X},\widetilde{T}|\pi,\{\widetilde{d},d\})\geq \overline{\rm mdim}_{M}(T,X,d).$$
\end{proof}

We collect several   useful  facts to  compare  $\omega$-weighted  metric mean dimension with $(\boldsymbol{\omega, 1-\omega})$-weighted metric mean dimension.

\begin{lem}\label{lem 3.9}
The following statements hold:

$(1)$ Let $(X,T)$ be a TDS. Then for any $\mu \in M(X,T)$, one has 
$$h_{\mu}(T,\UU)=\lim_{n \to \infty}\frac{1}{n}\inf_{\alpha \succ \UU^n}H_{\mu}(\alpha),$$
where  the infimum ranges over all finite Borel  partitions $\alpha$ of $X$ that  are finer than $\UU^n$.

$(2)$ Let $p_1,\cdots,p_n$ be non-negative numbers with $\sum_{i=1}^n p_i=1$. Then for any real numbers $x_1,\cdots,x_n$, 
$$\sum_{i=1}^n(-p_i\log p_i+p_ix_i)\le\log\sum_{i=1}^n e^{x_i}.$$

$(3)$  Let $\pi:(X,d_X,T) \rightarrow (Y,d_Y,S)$ be a factor map and $\textbf{a}=(a_1,a_2)$ with $a_1> 0$ and $a_2\geq 0$.
\begin{align*}
&\overline{\rm mdim}_M^{\textbf{a}}(T,X,\{d_X,d_Y\})\\
=&\limsup_{\epsilon \to 0}\frac{1}{\logf}\sup_{\mu \in E(X,T)}\{a_1 \inf_{\diam(\UU,d_X)\leq \epsilon}h_{\mu}(T,\UU)+a_2 \inf_{\diam(\VV,d_Y)\leq \epsilon}h_{\pi_{*}\mu}(S,\VV)\}.
\end{align*}
It is also valid  for  $\underline{\rm mdim}_M^{\textbf{a}}(T,X,\{d_X,d_Y\})$ by changing $\limsup_{\epsilon \to 0}$ into $\liminf_{\epsilon \to 0}$.
\end{lem}

\begin{proof}
(1). It follows from  \cite{hmry04, gw06}  (see also \cite[Theorem 4.14]{hyz11}).

(2).  An available proof is given in \cite[Lemma 9.9, p.217]{w82}.

(3).  By (\ref{equ 4.2}) and (\ref{equ 4.3}), for every $\epsilon >0$ we have 
\begin{align*}
h_{top}^{\textbf{a}}(T,X,\{d_X,d_Y\},\epsilon)
\leq& \sup_{\mu \in E(X,T)}\{\inf_{\diam (\UU,d_X)\leq \frac{\epsilon}{8}} \inf_{\diam (\VV,d_Y)\leq \frac{\epsilon}{8}}(a_1h_{\mu}(T,\UU)+a_2 h_{\pi_{*}\mu} (S,\VV))\},\\
\leq & h_{top}^{\textbf{a}}(T,X,\{d_X,d_Y\},\frac{\epsilon}{8}).
\end{align*}
This implies the desired variational principle.
\end{proof}
\begin{thm}\label{thm 4.8}
Let  $\pi:(X,d_X, T)\rightarrow (Y,d_Y,S)$ be a factor map. Then  for any $\omega \in (0,1]$,
\begin{itemize}
\item [(1)]$\overline{\rm mdim}_M^{\omega}(X,T|\pi,\{d_X,d_Y\})\leq\overline{\rm mdim}_M^{(\boldsymbol{\omega, 1-\omega})}({T},{X},\{d_X,d_Y\})+\omega \cdot  \overline{\rm mdim}_M({S},{Y},d_Y)$. 
\item [(2)]  $\overline{\rm mdim}_M^{\omega}(X,T|\pi,\{d_X,d_Y\})\geq\overline{\rm mdim}_M^{(\boldsymbol{\omega, 1-\omega})}({T},{X},\{d_X,d_Y\})$. 
\end{itemize}
\end{thm}

\begin{proof}
(1). 
Let  $\UU, \VV$ be two open covers of $X$ and $Y$ with $\diam(\UU,d_X)\leq \frac{\epsilon}{2}$ and  $\diam(\VV,d_Y)\leq \frac{\epsilon}{2}$, respectively. 
Then   the diameter of every element of $\VV^n$ with respect to $(d_Y)_n$ is at most $\frac{\epsilon}{2}$. Using this fact, choose a family  $\{V_1,..., V_{N(\VV^n)}\}$ of open sets $Y$ with $\diam(V_j,d_Y)<\epsilon$.  We have $$\#^{\omega}(\pi,n,\epsilon)\leq  \sum_{j=1}^{N(\VV^n)}(\#(\pi^{-1}(V_j),n,\epsilon))^{\omega} \leq N(\VV^n)\cdot (N(\UU^n))^{\omega}.$$
This yields that
\begin{align*}
\lim_{n \to \infty}\frac{\log  \#^{\omega}(\pi,n,\epsilon)}{n}
\leq &\omega (\inf_{\diam(\UU,d_X)
\leq \frac{\epsilon}{2}}h_{top}(T,\UU)+\inf_{\diam(\VV,d_Y)\leq \frac{\epsilon}{2}}h_{top}(S,\VV))\\
+&
(1-\omega)\cdot
\inf_{\diam(\VV,d_Y)\leq \frac{\epsilon}{2}}h_{top}(S,\VV).
\end{align*}
Hence, we   get the first inequality.

(2). Fix $\epsilon >0$ and $\mu \in E(X,T)$. Choose two open covers $\UU \in \CC_X^o, \VV\in \CC_Y^o$ such that $\diam(\UU,d_X)<\epsilon$, $\Leb(\UU,d_X)\geq \frac{\epsilon}{4}$,  and $\diam(\VV,d_Y)<\epsilon$, $\Leb(\VV,d_Y)\geq \frac{\epsilon}{4}$. Let $\alpha$ be a finite Borel  partition of $X$ that refines  $\UU^n$, and let $\beta$ be a finite Borel  partition of $Y$ that refines  $\VV^n$.  It is easy to show 
$$\omega H_{\mu}(\alpha)+(1-\omega)H_{\pi_{*}\mu}(\beta)=H_{\pi_{*}\mu}(\beta)+\omega H_{\mu}(\alpha |\pi^{-1}\beta).$$ For any $B\in \beta$, we  denote by
$$\alpha_B:=\{A\in \alpha: A\cap \pi^{-1}B\not=\emptyset\}$$
the atoms of $\alpha$ that have non-empty intersection with $\pi^{-1}B$.
Then  $\pi^{-1}B=\cup_{A\in \alpha_B}(A\cap \pi^{-1}B)$ and $H_{\mu}(\alpha |\pi^{-1}\beta)\leq \sum_{B\in \beta}\pi_{*}\mu(B)\log \#\alpha_B$.  Thus, by Lemma \ref{lem 3.9},(2) we have
\begin{align*}
\omega H_{\mu}(\alpha)+(1-\omega)H_{\pi_{*}\mu}(\beta)\leq  \sum_{B\in \beta}\pi_{*}\mu(B)(\log (\#\alpha_B)^{\omega}-\pi_{*}\mu(B))
\leq  \log {\sum_{B\in \beta}(\#\alpha_B)^{\omega}}.
\end{align*}
The arbitrariness of $\alpha $ and $\beta$  implies that
\begin{align}\label{inequ 5.7}
&\omega\inf_{\alpha \succ \UU^n}H_{\mu}(\alpha)+(1-\omega) \inf_{\beta \succ \VV^n}H_{\pi_{*}\mu}(\beta)\\
\leq& \log \inf_{\alpha \succ \UU^n,\atop \beta\succ \VV^n} {\sum_{B\in \beta}(\#\alpha_B)^{\omega}}
\leq  \log \inf_{\alpha \succ \UU^n, \beta\succ \VV^n,\atop \alpha \succ \pi^{-1}\beta } {\sum_{B\in \beta}(\#\alpha_B)^{\omega}} \nonumber.
\end{align}
Choose  an  open cover $\{V_j\}_{j=1}^m$ of $Y$ with $\diam(V_j, (d_Y)_n)<\frac{\epsilon}{8}$  for every $1\leq j\leq m$ such that
\begin{align}\label{inequ 5.8}
\#^{\omega}(\pi,n,\frac{\epsilon}{8})=\sum_{j=1}^m
(\#(\pi^{-1}(V_j),n,\frac{\epsilon}{8}))^{\omega}.
\end{align}
Using the open cover $\{V_j\}_{j=1}^m$,  we  construct a  partition $\beta=\{V_j^{'}\}_{j=1}^m$ of $Y$ with  $V_j^{'}\subset V_j$ and $\diam (V_j{'},(d_Y)_n)<\frac{\epsilon}{8}$ for every $1\leq j \leq m$. Then  $\beta \succ  \VV^n$ since $\Leb(\VV)\geq  \frac{\epsilon}{4}$. 
For each  $\pi^{-1}V_j{'}$, we  construct a finite partition  $\alpha_j$ of $\pi^{-1}V_j{'}$ with  $\#\alpha_j=\#(\pi^{-1}(V_j^{'}),n,\frac{\epsilon}{8})$ and  $\diam(A,(d_X)_n)<\frac{\epsilon}{8}$ for every  $A\in \alpha_j$. Then the family $$\alpha=\{A\}_{ A\in \alpha_j, 1\leq j\leq m}$$ is a partition of $X$ that refines $\UU^n$ since $\Leb(\UU)\geq \frac{\epsilon}{4}$, and $\alpha \succ \pi^{-1}(\beta)$.  Hence, by (\ref{inequ 5.7}) and (\ref{inequ 5.8}) we  obtain
\begin{align}\label{inequ 5.9}
\omega\inf_{\alpha \succ \UU^n}H_{\mu}(\alpha)+(1-\omega) \inf_{\beta \succ \VV^n}H_{\pi_{*}\mu}(\beta)\leq  \log \#^{\omega}(\pi,n,\frac{\epsilon}{8}).
\end{align}
Then, by Lemma \ref{lem 3.9},(1) we get
\begin{align*}
\sup_{\mu \in E(X,T)}\{\omega  \inf_{\diam(\UU,d_X)\leq \epsilon}h_{\mu}(T,\UU)+(1-\omega) \inf_{\diam(\VV,d_Y)\leq \epsilon}h_{\pi_{*}\mu}(S,\VV)\}
\leq  \lim_{n \to \infty}\frac{\log  \#^{\omega}(\pi,n,\frac{\epsilon}{8})}{n}.
\end{align*}
Using  Lemma \ref{lem 3.9},(3), this shows (2).
\end{proof}

We present  several  corollaries  as an immediate consequence of Theorem  \ref{thm 4.8}.

\begin{cor}\label{cor 3.11}
Let  $\pi:(X,d_X, T)\rightarrow (Y,d_Y,S)$ be a factor map. If $\overline{\rm mdim}_M({S},{Y},d_Y)=\underline{\rm mdim}_M({S},{Y},d_Y)$,   then for every  $\omega \in (0,1]$,
\begin{align*}
&\omega \cdot \overline{\rm mdim}_M({T},{X},d_X) +(1-\omega) {\rm mdim}_M({S},{Y},d_Y)\\
\leq&
\overline{\rm mdim}_M^{\omega}(X,T|\pi,\{d_X,d_Y\})
\\
\leq&   \omega \cdot \overline{\rm mdim}_M({T},{X},d_X) +{\rm mdim}_M({S},{Y},d_Y).
\end{align*}
\end{cor}


\begin{cor}\label{cor 4.9}
Let  $\pi:(X,d_X, T)\rightarrow (Y,d_Y,S)$ be a factor map. If $\overline{\rm mdim}_M({S},{Y},d_Y)=0$, then for any $\omega \in (0,1]$,
\begin{align*}
\overline{\rm mdim}_M^{\omega}(X,T|\pi,\{d_X,d_Y\})=  \overline{\rm mdim}_M^{(\boldsymbol{\omega, 1-\omega})}({T},{X},\{d_X,d_Y\})
=\omega \cdot\overline{\rm mdim}_M({T},{X},d_X).
\end{align*}
\end{cor}

\begin{rem}\label{rem 3.13}
For some factor maps between two dynamical systems, if the metric mean dimension of the factor system is positive,  it may happen that
$$\overline{\rm mdim}_M^{(\boldsymbol{\omega, 1-\omega})}({T},{X},\{d_X,d_Y\})<\overline{\rm mdim}_M^{\omega}(X,T|\pi,\{d_X,d_Y\}),$$
and  the weighted metric mean dimension $\overline{\rm mdim}_M^{\omega}(X,T|\pi,\{d_X,d_Y\})$ is  not a convex combination of the metric mean dimensions of the extension system and the factor system.

Indeed, considering the  Example \ref{ex 2,3},  and letting $s_1,s_2 \in (0,1)$ with $s_1<s_2$ and $id: (\sigma,X^{\mathbb{Z}},d_{s_1}^{\mathbb{Z}}) \rightarrow (\sigma,X^{\mathbb{Z}},d_{s_2}^{\mathbb{Z}}) $ be the indentity map.  For every $0<\omega< 1$, we have
\begin{align*}
\overline{\rm mdim}_M^{(\boldsymbol{\omega, 1-\omega})}({\sigma},{X^{\mathbb{Z}}},\{d_{s_1}^{\mathbb{Z}},d_{s_2}^{\mathbb{Z}}\})=&\omega \cdot  {{\rm mdim}}_M(\sigma,X^{\mathbb{Z}},d_{s_1}^{\mathbb{Z}})+ (1-\omega){{\rm mdim}}_M(\sigma,X^{\mathbb{Z}},d_{s_2}^{\mathbb{Z}})\\
=&\omega s_1+(1-\omega)s_2.
\end{align*} 
Let $\{V_j\}_{j=1}^m$ be an  open cover  of $X^{\mathbb{Z}}$ with $\diam(V_j, d_{s_2}^{\mathbb{Z}})<\epsilon$. Notice that $d_{s_1}^{\mathbb{Z}}\leq d_{s_2}^{\mathbb{Z}}$. Then  $\diam(id^{-1}(V_j), d_{s_1}^{\mathbb{Z}})<\epsilon$, and hence $\#^{\omega}(\pi,n,\epsilon)
$ coincides with the smallest cardinality of   a family of open cover  $\{V_j\}_{j=1}^m$  of $X^{\mathbb{Z}}$ with   $\diam(V_j, d_{s_2}^{\mathbb{Z}})<\epsilon$.   Thus, for every $0<\omega <1$ we have 
$$\overline{\rm mdim}_M^{\omega}(X^{\mathbb{Z}},\sigma|\pi,\{d_{s_1}^{\mathbb{Z}},d_{s_2}^{\mathbb{Z}}\})=s_2.$$
This means that 
$\overline{\rm mdim}_M^{(\boldsymbol{\omega, 1-\omega})}({\sigma},{X^{\mathbb{Z}}},\{d_{s_1}^{\mathbb{Z}},d_{s_2}^{\mathbb{Z}}\})< \overline{\rm mdim}_M^{\omega}(X^{\mathbb{Z}},\sigma|\pi,\{d_{s_1}^{\mathbb{Z}},d_{s_2}^{\mathbb{Z}}\}).$

\end{rem}

Notice that the $0$-weighted metric mean dimension  is exactly the metric mean dimension of $Y$. By Corollary \ref{cor 3.11}, the following corollary shows how the $\omega$-weighted metric mean dimension  varies with the parameter $\omega$ by letting $\omega \to 0$.  

\begin{cor}\label{cor 4.11}
Let  $\pi:(X,d_X, T)\rightarrow (Y,d_Y,S)$ be a factor map. If $\overline{\rm mdim}_M({S},{Y},d_Y)=\underline{\rm mdim}_M({S},{Y},d_Y)$ and  $\overline{\rm mdim}_M({T},{X},d_X)<\infty$, then  as  $\omega \to 0$,
$$\overline{\rm mdim}_M^{\omega}(X,T|\pi,\{d_X,d_Y\}) \rightarrow {\rm mdim}_M(S,Y,d_Y).$$
\end{cor}

\subsection{Feng-Huang's   weighted Bowen topological entropy}

Feng and Huang  \cite{fw1616} introduced the weighted Bowen topological entropy of factor maps, which  resembles the  Hausdorff dimension  in  fractal geometry.

Firstly, we review   the  definition of Bowen metric mean dimension defined by  Carath\'eodory-Pesin structures \cite{p97}.  Let $(X,d,T)$ be a TDS.  For  $s\geq 0$ and $N\in \mathbb{N}$, we put
$$\mathcal{M}_{N,s}(X,\UU)=\inf\{\sum_{i\in I}e^{-n_i s}\},$$
where the infimum ranges over all  finite or countable  families  $\{A_i\}_{i\in I}$  satisfying $X=\cup_{i\in I}A_i$,  $A_i\in \UU^{n_i}$ and  $n_i\geq N$ for all $i\in I$.

Let $\mathcal{M}_{s}(X,\UU)=\lim\limits_{N\to \infty}\mathcal{M}_{N,s}(X,\UU)$. Then  there is  a critical value  of the parameter $s$, denoted by $h_{top}^B(T,X,\UU)$, such that $\mathcal{M}_{s}(X,\UU)$  jumps  from  $\infty$ to $0$, i.e.,
\begin{align*}
h_{top}^B(T,X,\UU):
=\inf\{s>0: \mathcal{M}_{\mathcal{U},s}(Z)=0\}
=\sup\{s>0: \mathcal{M}_{\mathcal{U},s}(Z)=\infty\}.
\end{align*}

The \emph{Bowen upper  metric mean dimension of $X$}  is defined by 
$$\overline{\rm mdim}_M^B(T,X,d)=\limsup_{\epsilon \to 0} \frac{1}{\logf}\inf_{\diam (\UU,d) \leq \epsilon}h_{top}^B(T,X,\UU).$$
By \cite[Proposition  3.4, (c)]{ycz22}, we have $\overline{\rm mdim}_M^B(T,X,d)=\overline{\rm mdim}_M(T,X,d).$

Next we recall Feng-Huang's  weighted Bowen topological entropy \cite{fw1616}.  Let $(X,d_X,T)$ and $(Y,d_Y,S)$ be two TDSs, and  let $\pi: X\rightarrow Y$  be a factor map. Given $\UU \in \CC_X^o$ and $\VV\in \CC_Y^o$,  the diameter  of $\{\UU,\VV\}$ is defined by $\diam (\UU,\VV):=\max\{\diam(\UU,d_X),\diam (\VV,d_Y)\}.$

Suppose that the weight $\textbf{a}=(a_1,a_2)$ with $a_1>0$ and $a_2\geq 0$ is given.   Each element $A$ in $\UU^{\lceil a_1n \rceil} \vee \pi^{-1} \VV^{\lceil (a_1+a_2)n \rceil}$  can be expressed as  the following form:
\begin{align}\label{equ 4.6}
A=(\bigcap_{j=0}^{\lceil a_1n \rceil-1}T^{-j}U_{i_j})  \bigcap (\bigcap_{j=0}^{\lceil (a_1+a_2)n \rceil-1}\pi^{-1}S^{-j}{V}_{i_j}),
\end{align}
where $U_{i_j}\in \UU$ for $j=0,..., \lceil a_1n \rceil -1 $, and  $V_{i_j}\in \VV$ for $j=0,..., \lceil (a_1+a_2)n \rceil-1 $. Here, $\lceil x \rceil$ denotes the  smallest integer  $\geq x$.  The  geometric length of $A$ is defined by  $m(A)=n$.  

Let $s\geq 0$ and  $N\in \mathbb{N}$.  Set
$$\mathcal{M}_{N,s}^{\textbf{a}}(X,\{\UU,\VV\})=\inf\{\sum_{i\in I}e^{-m(A_i)s}\},$$
where the infimum ranges over all  finite or countable  family  $\{A_i\}_{i\in I}$ satisfying  $X= \cup_{i\in I}A_i$, $A_i \in \UU^{\lceil a_1n \rceil-1}\vee \pi^{-1} \VV^{\lceil a_1+a_2 n \rceil-1}$ and $m(A_i)\geq N$ for all $i\in I$.

Since $\mathcal{M}_{N,s}^{\textbf{a}}(X,\{\UU,\VV\})$ is non-decreasing in $N$,  the limit $$\mathcal{M}_{s}^{\textbf{a}}(X,\{\UU,\VV\})=\lim\limits_{N\to \infty}\mathcal{M}_{N,s}^{\textbf{a}}(X,\{\UU,\VV\})$$ exists.  One
 can check  that there is a  critical value  of  the parameter $s$  such  that $\mathcal{M}_{s}^{\textbf{a}}(X,\{\UU,\VV\})$ jumps  from $\infty$ to $0$.  The  critical value is defined by 
\begin{align*}
h_{top}^{\textbf{a},B}(T,X,\{\UU,\VV\}):&=\inf\{s>0:\mathcal{M}_{s}^{\textbf{a}}(X,\{\UU,\VV\})=0\}\\
&=\sup\{s>0:\mathcal{M}_{s}^{\textbf{a}}(X,\{\UU,\VV\})=\infty\}.
\end{align*}

The \emph{\textbf{a}-weighted Bowen topological entropy of  $X$} is defined by 
$$h_{top}^{\textbf{a},B}(\pi, T)=\sup_{(\UU,\VV)}h_{top}^{\textbf{a},B}(T,X,\{\UU,\VV\}),$$
where the supremum ranges over all pairs $(\UU, \VV)$ with $\UU \in \mathcal{C}_X^o$ and $\VV \in \mathcal{C}_Y^o$.

\begin{df}
We define  the  \textbf{a}-\emph{weighted Bowen upper  and lower metric mean dimensions of $X$} as
\begin{align*}
\overline{\rm {mdim}}_M^{\textbf{a},B}(T,X,\{d_X,d_Y\})&=\limsup_{\epsilon \to 0}\frac{1}{\log \frac{1}{\epsilon}} \inf_{\diam (\UU,\VV)\leq \epsilon}h_{top}^{\textbf{a},B}(T,X,\{\UU,\VV\}),\\
\underline{\rm {mdim}}_M^{\textbf{a},B}(T,X,\{d_X,d_Y\})&=\liminf_{\epsilon \to 0}\frac{1}{\log \frac{1}{\epsilon}} \inf_{\diam (\UU,\VV)\leq \epsilon}h_{top}^{\textbf{a},B}(T,X,\{\UU,\VV\}),
\end{align*}
respectively.
\end{df}
In particular, Bowen metric mean dimension is exactly  the weighted Bowen metric mean dimension by letting $\textbf{a}=(1,0)$ and $Y$ be a single point. We remark that the  \textbf{a}-weighted Bowen  metric mean dimension depends on the compatible metrics on  $X$ and $Y$, and the weight. Compared with the aforementioned two weighted metric mean dimensions,  this kind of \textbf{a}-weighted  metric mean dimension  is   more difficult  to  compute its precise value.   Hence,  we  introduce \textbf{a}-\emph{weighted  upper and lower capacity  metric mean dimensions}  to estimate the upper bound of \textbf{a}-weighted Bowen  metric mean dimension.

We define the \textbf{a}-weighted capacity topological entropy of $X$ w.r.t. $\{\UU,\VV\}$  as
$$h_{top}^{\textbf{a},UC}(T,X,\{\UU,
\VV\})=\limsup_{n \to \infty}\frac{1}{n}\log N(\UU^{\lceil a_1n \rceil} \vee \pi^{-1} \VV^{\lceil (a_1+a_2)n \rceil}).$$

\begin{df}\label{df 4.14}
The \textbf{a}-\emph{weighted  upper and lower capacity  metric mean dimensions  of $X$} are respectively given by
\begin{align*}
\overline{\rm {mdim}}_M^{\textbf{a},UC}(T,X,\{d_X,d_Y\})&=\limsup_{\epsilon \to 0}\frac{1}{\logf}\inf_{\diam (\UU,\VV)\leq \epsilon} h_{top}^{\textbf{a},UC}(T,X,\{\UU,
\VV\}),\\
\underline{\rm {mdim}}_M^{\textbf{a},UC}(T,X,\{d_X,d_Y\})&=\liminf_{\epsilon \to 0}\frac{1}{\logf}\inf_{\diam (\UU,\VV)\leq \epsilon} h_{top}^{\textbf{a},UC}(T,X,\{\UU,
\VV\}).
\end{align*}
\end{df}

An  equivalent  formulation for  Definition \ref{df 4.14}  is  as follows. For each $n\in \mathbb{N}$, the  \textbf{a}-\emph{weighted Bowen metric on $X$}  is given by
$$D_n^{\textbf{a}}(x,y):=\max\{(d_X)_{\lceil a_1n \rceil}(x,y), (d_Y)_{\lceil (a_1+a_2)n \rceil}(\pi x,\pi y)\},$$
where
$(d_X)_{\lceil a_1n \rceil}(x,y)=\max_{0\leq j\leq \lceil a_1n \rceil-1}d_X(T^jx,T^jy)$ 
and 
$$(d_Y)_{\lceil (a_1+a_2)n \rceil}(\pi x,\pi y)=\max_{0\leq j\leq \lceil (a_1+a_2)n \rceil-1}d_Y(S^j(\pi x),S^j(\pi y)).$$ 

Then it is easy to   show that
\begin{align*}
\overline{\rm {mdim}}_M^{\textbf{a},UC}(T,X,\{d_X,d_Y\})&=\limsup_{\epsilon \to 0}\frac{1}{\logf }\limsup_{n \to \infty}\frac{\log r(X, D_n^{\textbf{a}},\epsilon) }{n},\\
\underline{\rm {mdim}}_M^{\textbf{a},UC}(T,X,\{d_X,d_Y\})&=\liminf_{\epsilon \to 0}\frac{1}{\logf }\limsup_{n \to \infty}\frac{\log r(X, D_n^{\textbf{a}},\epsilon) }{n}.
\end{align*}


\begin{prop}\label{prop 4.14}
Let $\pi:(X,d_X,T)\rightarrow (Y,d_Y,S)$ be a factor map, and  \rm {\textbf{a}}$=(a_1,a_1)$  with $a_1>0$ and $a_2\geq 0$. Then 
\begin{align*}
\overline{\rm {mdim}}_M^{\textbf{a},B}(T,X,\{d_X,d_Y\})&\leq \overline{\rm {mdim}}_M^{\textbf{a},UC}(T,X,\{d_X,d_Y\}),\\
\underline{\rm {mdim}}_M^{\textbf{a},B}(T,X,\{d_X,d_Y\})&\leq \underline{\rm {mdim}}_M^{\textbf{a},UC}(T,X,\{d_X,d_Y\}).
\end{align*}
\end{prop}

\begin{proof}
Let $s> h_{top}^{\textbf{a},UC}(T,X,\{\UU,
\VV\})$. Then for sufficiently large $n$, one has $N(\UU^{\lceil a_1n \rceil} \vee \pi^{-1} \VV^{\lceil (a_1+a_2)n \rceil})<e^{ns}$. If  the family $\{A_i\}_{i=1}^m   \subset U^{\lceil a_1n \rceil} \vee \pi^{-1} \VV^{\lceil (a_1+a_2)n \rceil}$ is a subcover of $X$ satisfying $m=N(\UU^{\lceil a_1n \rceil} \vee \pi^{-1} \VV^{\lceil (a_1+a_2)n \rceil})$, then  $X= \cup_{j=1}^mA_j$ and $m(A_j)=n$ for all $1\leq j\leq m$.  Using this fact, one can show  $h_{top}^{\textbf{a},B}(T,X,\{\UU,\VV\})\leq s$, and hence $h_{top}^{\textbf{a},B}(T,X,\{\UU,\VV\})\leq h_{top}^{\textbf{a},UC}(T,X,\{\UU,
\VV\})$.  This implies the desired inequalities.   
\end{proof}

We compare the weighted Bowen metric mean dimension  with  the metric mean dimensions of factor system and extension system.

\begin{thm}\label{thm 4.13}
Let $\pi:(X,d_X,T)\rightarrow (Y,d_Y,S)$ be a factor map, and \rm {\textbf{a}}$=(a_1,a_1)$   with $a_1>0$ and $a_2\geq 0$. Then 
\begin{align*}
&\max\{a_1\overline{\rm mdim}_M(T,X,d_X), (a_1+a_2)\overline{\rm mdim}_M(S,Y,d_Y)\}\\
\leq &\overline{\rm {mdim}}_M^{\textbf{a},B}(T,X,\{d_X,d_Y\})
\\
\leq&a_1\overline{\rm mdim}_M(T,X,d_X) + (a_1+a_2)\overline{\rm mdim}_M(S,Y,d_Y).
\end{align*}
\end{thm}

\begin{proof}
Fix $s> 0$ and $N \in \mathbb{N}$.  Let $\{A_i\}_{i\in I}$ be a finite or countable  family satisfying  $X= \cup_{i\in I}A_i$ and $m(A_i)\geq N$ for all $i\in I$, where each $A_i$ has the form given in  (\ref{equ 4.6}). Then $X$ is also covered by a  finite or countable  family $\{B_i\}_{i\in I}$ with $m(B_i)= \lceil a_1m(A_i)\rceil \geq \lceil a_1N\rceil$ for all $i\in I$. By definitions,  we obtain 
$$\mathcal{M}_{\lceil a_1N\rceil,s}(X,\UU)\leq \sum_{i \in I}e^{- \lceil a_1m(A_i)\rceil s} \leq\sum_{i \in I}e^{-  m(A_i) (a_1s)}.$$
This yields that  $\mathcal{M}_{\lceil a_1N\rceil,s}(X,\UU)\leq  \mathcal{M}_{N,(a_1s)}^{\textbf{a}}(X,\{\UU,\VV\})$. So
$$a_1\overline{\rm mdim}_M(T,X,d_X)\leq \overline{\rm {mdim}}_M^{\textbf{a},B}(T,X,\{d_X,d_Y\}). $$
The similar argument gives  us $$(a_1+a_2)\overline{\rm mdim}_M(S,Y,d_Y)\leq \overline{\rm {mdim}}_M^{\textbf{a},B}(T,X,\{d_X,d_Y\}).$$ 

Now fix  $\UU \in \CC_X^o, \VV\in \CC_Y^o$. Then 
$$h_{top}^{\textbf{a},UC}(T,X,\{\UU,
\VV\}) \leq  a_1 h_{top}(T,\UU) +(a_1+a_2) h_{top}(S,\VV).$$
Using Proposition \ref{prop 4.14}, we  get $$\overline{\rm {mdim}}_M^{\textbf{a},B}(T,X,\{d_X,d_Y\})
\leq a_1\overline{\rm mdim}_M(T,X,d_X) + (a_1+a_2)\overline{\rm mdim}_M(S,Y,d_Y).$$
\end{proof}

\begin{rem}\label{rem 3.19}
In general, if the factor system has positive metric mean dimension, one cannot  expect that the weighted Bowen metric mean dimension can be expressed as a  convex combination of the metric mean dimensions of the factor system and extension system.

Indeed, considering the example given in Remark  \ref{rem 3.13}, and taking $\textbf{a}=(\omega,1-\omega)$ with $0<\omega <1$,  by Theorem \ref{thm 4.13} we have
$\overline{\rm mdim}_M^{(\boldsymbol{\omega, 1-\omega}), B}({\sigma},{X^{\mathbb{Z}}},\{d_{s_1}^{\mathbb{Z}},d_{s_2}^{\mathbb{Z}}\}) \geq s_2.$
On the  other hand, let
$$D_n^{\textbf{a}}(x,y):=\max\{(d_{s_1}^{\mathbb{Z}})_{\lceil \omega n \rceil}(x,y), (d_{s_2}^{\mathbb{Z}})_{ n }(x,y)\}=(d_{s_2}^{\mathbb{Z}})_{ n }(x,y).$$
Then   a $((d_{s_2}^{\mathbb{Z}})_{ n } ,\epsilon)$-spanning set  of $X^{\mathbb{Z}}$ is also a $(D_n^{\textbf{a}},\epsilon)$-spanning set of $X^{\mathbb{Z}}$.  Using Proposition \ref{prop 4.14}, we have
$\overline{\rm mdim}_M^{\textbf{($\omega,  1-\omega$)}, B}({\sigma},{X^{\mathbb{Z}}},\{d_{s_1}^{\mathbb{Z}},d_{s_2}^{\mathbb{Z}}\}) \leq  s_2.$ 
This implies that $$\overline{\rm mdim}_M^{(\boldsymbol{\omega, 1-\omega}), B}({\sigma},{X^{\mathbb{Z}}},\{d_{s_1}^{\mathbb{Z}},d_{s_2}^{\mathbb{Z}}\})=s_2\not= \omega \cdot  {{\rm mdim}}_M(\sigma,X^{\mathbb{Z}},d_{s_1}^{\mathbb{Z}})+ (1-\omega)\cdot {{\rm mdim}}_M(\sigma,X^{\mathbb{Z}},d_{s_2}^{\mathbb{Z}}).$$
\end{rem}

Now  we are ready to prove Theorem \ref{thm 1.2}.

\begin{proof}[Proof of Theorem \ref{thm 1.2}]

It  is a direct consequence of   Corollary \ref{cor 4.9} and Theorem \ref{thm 4.13}.
\end{proof}

\begin{que}
 Let  $\pi:(X,d_X,T)\rightarrow (Y,d_Y,S)$ be a factor map. Suppose that  $\overline{\rm mdim}_M(S,Y,d_Y)>0$.  Does
$$\overline{\rm {mdim}}_M^{\textbf{($\omega$, $1-\omega$)},B}(T,X,\{d_X,d_Y\})=\overline{\rm mdim}_M^{\omega}(X,T|\pi,\{d_X,d_Y\})$$
hold  for all  $0<\omega <1$.
\end{que}

\section{relative  conditional metric mean dimension of factor maps}\label{sec 4}

In this section, we  introduce the relative  conditional metric mean dimension of factor maps, and prove Theorem \ref{thm 1.3}. 


Let $\pi:(X,T)\rightarrow (Y,S)$ be a  factor map.  Let $\UU, \mathcal{V}\in \mathcal{C}_X^o$. Put
$$N(\UU|\mathcal{V}\vee\pi):=\max\{N(\UU, V\cap \pi^{-1}(y)): V\in \mathcal{V}, y\in Y\},$$
and  let  $N(\UU|\pi):= \max_{y \in Y} N(\UU^n,\pi^{-1}(y))$.
Since $\{ \log N(\UU^n|\mathcal{V}^n\vee\pi)\}$ is a sub-additive sequence in $n$, the following limit
$$h_{top}(T,\UU|\mathcal{V}\vee\pi):=\lim_{n \to \infty}\frac{1}{n}\log N(\UU^n|\mathcal{V}^n\vee\pi)$$
exists.
If $\VV=\{X\}$, we let $h(T,\UU|\pi):= h_{top}(T,\UU|\{X\}\vee\pi)$.
 
The  \emph{relative topological conditional entropy of $X$} \cite{zz23}  is defined by
$$h^{*}_{top}(X,T|\pi)=\inf_{ \mathcal{V}\in \mathcal{C}_X^o}\sup_{\mathcal{U}\in \mathcal{C}_X^o}h_{top}(T,\UU|\mathcal{V}\vee\pi).$$

If $Y=\{y_0\}$ consists of a single point, this definition reduces to  the topological conditional entropy of $X$ \cite{mis76}; if the infimum  only takes the  open cover $\mathcal{V}=\{X\}$, the above definition reduces to  the  relative topological entropy \cite{lw97}. Hence, the relative topological conditional entropy is a fusion of this  two kind of entropies. 

Notice that,  if  $h^{*}_{top}(X,G|\pi)=\infty$, then for  $\mathcal{V}\in \mathcal{C}_X^o$, we have  $\sup_{\mathcal{U}\in \mathcal{C}_X^o}h_{top}(G,\UU|\mathcal{V}\vee\pi)=\infty$. This motivates us to introduce the following notion:

\begin{df}\label{df 4.17}
Let $\pi:(X,d_X,T)\rightarrow (Y,S)$ be a factor map. We respectively define the  relative  conditional  upper and lower  metric mean dimensions of  $\pi$ as
\begin{align*}
\overline{\rm mdim}_M^{*}(X,T|\pi,d_X)&=\limsup_{\epsilon \to 0}\frac{1}{\logf}\inf_{\diam (\UU, d_X)\leq \epsilon}h_{top}(T,\UU|\mathcal{V}\vee\pi), \\
\underline{\rm mdim}_M^{*}(X,T|\pi,d_X)&=\liminf_{\epsilon \to 0}\frac{1}{\logf}\inf_{\diam (\UU, d_X)\leq \epsilon}h_{top}(T,\UU|\mathcal{V}\vee\pi),
\end{align*}
where the infimum ranges over all finite open covers of $X$ with diameter at most $\epsilon$.
\end{df}
In particular,  by $\overline{\rm mdim}_M^{*}(T,X,d_X)$  we denote  the  \emph{topological conditional upper metric mean dimension of $X$} if  $Y$ is a single point, and  by $\overline{\rm mdim}_M(X,T|\pi,d_X)$ we  denote the \emph{relative  upper metric mean dimension  of $\pi$} if  $\mathcal{V}=\{X\}$ in  Definition \ref{df 4.17}.

The following theorem  justifies  that  $\overline{\rm mdim}_M^{*}(X,T|\pi,d_X)$ is well-defined in Definition \ref{df 4.17}, and is independent of the choice of $\VV \in \mathcal{C}_Y^o$ and the compatible metrics on $Y$.

\begin{thm}[=Theorem \ref{thm 1.3}]\label{thm 4.18}
Let $\pi:(X,d_X, T)\rightarrow (Y, S)$ be a factor map. Then
$$\overline{\rm mdim}_M^{*}(X,T|\pi,d_X)=\overline{\rm mdim}_M(X,T|\pi,d_X).$$
Consequently, we have $\overline{\rm mdim}_M^{*}(T,X,d_X)=\overline{\rm mdim}_M(T,X,d_X).$

The results are also valid for the corresponding lower metric mean dimensions.
\end{thm}
\begin{proof}
We only need to show that for every $\mathcal{V}\in \mathcal{C}_X^o$,
\begin{align}\label{equ 4.1}
\limsup_{\epsilon \to 0}\frac{1}{\logf}\inf_{\diam (\UU, d_X)\leq \epsilon}h_{top}(T,\UU|\mathcal{V}\vee\pi)=\overline{\rm mdim}_M(X,T|\pi,d_X),
\end{align}
since the  results for the corresponding lower metric mean dimensions are similarly obtained.

Fix $\mathcal{V}\in \mathcal{C}_X^o$.
It is clear that  $$\limsup_{\epsilon \to 0}\frac{1}{\logf}\inf_{\diam (\UU, d_X)\leq \epsilon}h_{top}(T,\UU|\mathcal{V}\vee\pi)\leq \overline{\rm mdim}_M(X,T|\pi,d_X).$$ 
For every $\mathcal{U}\in \mathcal{C}_X^o$ and $n \in \mathbb{N}$,  it holds that 
$$N(\UU^n|\pi)\leq N(\mathcal{V}^n|\pi)\cdot N(\UU^n|\mathcal{V}^n\vee \pi).$$
This implies that $$h_{top}(T,\UU|\pi)\leq h_{top}(T,\VV|\pi)+h_{top}(T,\UU|\mathcal{V}\vee\pi). $$ 
 Since $h_{top}(T,\VV|\pi) \leq h_{top}(T,\VV)$  is  finite, taking the infimum $\inf_{\diam (\UU, d_X)\leq \epsilon}\limits$ in the both sides of the above inequality, this yields the reverse inequality $$\overline{\rm mdim}_M(X,T|\pi,d_X)\leq \limsup_{\epsilon \to 0}\frac{1}{\logf}\inf_{\diam (\UU, d_X)\leq  \epsilon}h_{top}(T,\UU|\mathcal{V}\vee\pi).$$
\end{proof}

\begin{rem}
The proof of Theorem \ref{thm 4.18} also shows that  the complexity provided by the open cover  $\VV$ of $X$
only contributes a finite amount of topological entropy, and hence
does not lead to an increase of the relative metric mean dimension. This explains why the relative  metric mean dimension coincide with  the relative  conditional metric mean dimension.
\end{rem}

\section{An Abramov-Rokhlin formula for  random  metric mean dimension}\label{sec 5}

In this section,  we introduce random average  metric mean dimension  for  a  certain Lipschitz factor map of  random dynamical systems, and prove Theorem \ref{thm 1.4}. 


\subsection{Random  metric mean dimension and examples} 

We begin by recalling the setup of random  dynamical systems.
Let $(X, d)$ be a  compact metric space, and let $(\Omega, \mathcal{F}, \mathbb{P}, \theta)$ be  a measure-preserving system that preserves a $\theta$-invariant probability measure $\mathbb{P}$ on $\Omega$.  Equip $\Omega \times X$ with the product algebra $\mathcal{F}\otimes \mathcal{B}_X$.

A map  $T:  \mathbb{Z}\times \Omega \times X \rightarrow X$ is called a \emph{random  dynamical system} over the measure-preserving system $(\Omega, \mathcal{F}, \mathbb{P}, \theta)$ if  $T$ is measurable such that  the homeomorphisms $\{T_{\omega}^n: X\rightarrow X\}_{n\in \mathbb{Z},\omega\in \Omega}$ satisfy the cocycle property: for all  $\omega \in \Omega$, $m,n \in \mathbb{Z}$,  $$T_{\omega}^{n+m}(x)=T_{\theta^n \omega}^{m}\circ T_{\omega}^{n}(x),$$ where $T_{\omega}=T(1,\omega, \cdot)$ and the iteration is given by
\begin{align*}
T_{\omega}^n:=
\begin{cases}
T_{\sigma^{n-1}\omega}\circ\cdots\circ T_{\sigma\omega} \circ T_{\omega},  &\mbox{for}~n\geq 1\\
id,&\mbox{for}~n=0\\
(T_{\sigma^{n}\omega})^{-1}\circ\cdots\circ (T_{\sigma^{-2}\omega})^{-1}\circ (T_{\sigma^{-1}\omega})^{-1},  &\mbox{for}~n\leq-1.
\end{cases}
\end{align*}

In particular, if $\Omega=\{\omega_0\}$ is  exactly a single point, the above random  dynamical system  is reduced to  a  TDS $(X,d,T_{\omega_0})$. For every fixed $\omega \in \Omega$, $\{T_\omega^n\}_n$ is a family of homeomorphisms from $X$ to $X$. We call $(X,\{T_\omega^n\})$ a  non-autonomous dynamical system. The measure-preserving system $(\Omega, \mathcal{F}, \mathbb{P}, \theta)$ is  called  a \emph{driving system} of  $(X,T)$. One can associate  a   \emph{skew product transformation} $\Theta$ on $\Omega\times X$  by  letting  $\Theta(\omega,x)=(\theta\omega,T_{\omega}(x))$.   The  theory of random dynamical dynamics can be found in the  Arnold's monograph \cite{arn98}.  


To understand  the statistical and dynamical properties of  random dynamical systems, the concept of  random topological entropy was introduced in  \cite{kif86,bog92}.
For $n  \in \mathbb{N}$ and  $\omega \in \Omega$,  we define  a family $\{d_n^\omega\}_{n\geq 1,\omega \in \Omega}$ of  Bowen metrics  on $X$ by 
$$d_{n}^{\omega}(x,y)=\max_{0\leq j\leq n-1 }d(T_{\omega}^jx, T_{\omega}^jy),$$ 
where $x,y \in X$. Denote by  $s_n(X,d,\omega,\epsilon)$, $r_n(X,d,\omega,\epsilon)$ the maximal  cardinality  of {$(d_n^\omega,\epsilon)$-separated sets} of $X$, and the  minimal  cardinality  of {$(d_n^\omega,\epsilon)$-spanning sets} of $X$, respectively.

The \emph{random  topological entropy of $X$} is defined by
\begin{align*}
h_{\mathbb{P}}^{(r)}(T,X)&=\lim_{\epsilon \to 0}\limsup_{n\to \infty}\frac{ 1}{n}\int \log r_{n}(X,d,\omega,\epsilon) d \mathbb{P}(\omega),\\
&=\lim_{\epsilon \to 0}\limsup_{n\to \infty}\frac{1}{n}\int \log s_{n}(X,d,\omega,\epsilon) d \mathbb{P}(\omega).
\end{align*}
The random topological entropy of $X$ depends on the invariant measure $\mathbb{P}$ of the driving system,  while is independent of the choice of  the   compatible metrics on  $X$.

\begin{df}\label{df 4.21}
The  random  upper metric mean dimension  of $X$ \cite{myc19, wcy25} is defined by 
\begin{align*}
\mathbb{E}_{\mathbb{P}}\overline{{\rm mdim}}_M(T,X,d)&=\limsup_{\epsilon \to 0}\frac{1}{\logf}\limsup_{n\to \infty}\frac{ 1}{n}\int \log r_{n}(X,d,\omega,\epsilon) d \mathbb{P}(\omega),\\
&=\limsup_{\epsilon \to 0}\frac{1}{\logf}\limsup_{n\to \infty}\frac{1}{n}\int \log s_{n}(X,d,\omega,\epsilon) d \mathbb{P}(\omega).
\end{align*}
\end{df}

Similarly, the  random  lower metric mean dimension of $X$, denoted by $\mathbb{E}_{\mathbb{P}}\underline{{\rm mdim}}_M(T,X,d)$, is defined by changing $\limsup_{\epsilon \to 0}$  into  $\liminf_{\epsilon \to 0}$. 
We provide two examples of random dynamical systems and estimate their  random  upper and lower metric mean dimensions.

\begin{ex}\label{ex 5.3}
Let $(\Omega, \theta)$ be a TDS with a  $\theta$-invariant  Borel probability measure $\mathbb{P}$ on $\Omega$, and let $(X, ||\cdot||_{X})$ be a  compact real normed linear  space (e.g. $k$-dimensional torus $\mathbb{R}^k/\mathbb{Z}^k$).  Then the product space $X^{\mathbb{Z}}$, endowed  with the usual addition and multiplication operations,  is  a  compact normed linear space  in the sense  of the following norm:
$$||x-y||=\sum_{n\in \mathbb{Z}}\frac{||x_n-y_n||_X}{2^{|n|}},$$
where  $x=(x_n)_{n\in \mathbb{Z}}, y= (y_n)_{n\in \mathbb{Z}}\in X^{\mathbb{Z}}$.

Let $\sigma$ be the left shift map on $X^{\mathbb{Z}}$. Then it is a continuous linear operator since
$||\sigma(x)-\sigma (y)||\leq 2 ||x-y||.$ Hence, $(X^{\mathbb{Z}}, ||\cdot||, \sigma)$ is a TDS. Assume that $h:\Omega \rightarrow X^{\mathbb{Z}}$  is a measurable map. We define a family of continuous disturbances $\{T_{\omega}\}_{\omega}$ of $\sigma$ by letting $$T_{\omega}(x):=\sigma(x)+h(\omega)$$ for all $\omega \in \Omega$. Observe  that for any $1\leq j <n$,
$$T_{\omega}^j(x)=\sigma^j(x)+\sum_{m=0}^{j-1}\sigma^{j-1-m}h(\theta^m \omega).$$
It turns out that for every $n\in \mathbb{N}$ and $\omega\in \Omega$, $$||x-y||_{n}^{\omega}:=\max_{0\leq j <n}||T_{\omega}^j(x)-T_{\omega}^j(y)||=\max_{0\leq j <n}||\sigma^j(x)-\sigma^j(y)||=||x-y||_n.$$
Then, using  (\ref{equ 2.1}),  for any $\mathbb{P}\in M(\Omega, \theta)$ we have 
\begin{align*}
\mathbb{E}_{\mathbb{P}}\overline{{\rm mdim}}_M(T,X^{\mathbb{Z}},||\cdot||)=\overline{\rm mdim}_M(\sigma, X^{\mathbb{Z}},||\cdot||)=\overline{\rm dim}_B(X, ||\cdot||_X),\\
\mathbb{E}_{\mathbb{P}}\underline{{\rm mdim}}_M(T,X^{\mathbb{Z}},||\cdot||)=\underline{\rm mdim}_M(\sigma, X^{\mathbb{Z}},||\cdot||)=\underline{\rm dim}_B(X, ||\cdot||_X).
\end{align*}
\end{ex}

\begin{ex}\label{ex 5.4}
Let  $\Omega=\{1,2\}^{\mathbb{Z}}$,  and let $\theta: \Omega \rightarrow \Omega$ be the left shift map. Then $\theta$ preserves an ergodic Bernoulli probability measure $\mathbb{P}$ generated by a probability vector $(\frac{1}{2},\frac{1}{2})$ that  assigns $\frac{1}{2}$ to  each symbol equivalently. Suppose that $(X,d)$ is a compact  metric space. The product space  $X^{\mathbb{Z}}$ is metrizable by  $$D(x,y)=\sum_{n\in \mathbb{Z}}\frac{d(x_n,y_n)}{2^{|n|}}.$$ 
 $T_1:=\sigma$ and $T_2:=\sigma^2$ are  two  left shifts defined  on $X^{\mathbb{Z}}$.  Given  $\omega=(\omega_n)_{n\in \mathbb{Z}} \in \Omega$,  we define a continuous self-map on $X^{\mathbb{Z}}$ given by $T_{\omega}(x)=T_{\omega_0}(x)$. Then the random metric mean  dimension of $X$ is bounded by its box dimension:
\begin{align*}
\overline{\rm dim}_B(X, d)\leq \mathbb{E}_{\mathbb{P}}\overline{{\rm mdim}}_M(T,X^{\mathbb{Z}},D)\leq \frac{3}{2}\overline{\rm dim}_B(X, d),\\
\underline{\rm dim}_B(X, d)\leq \mathbb{E}_{\mathbb{P}}\underline{{\rm mdim}}_M(T,X^{\mathbb{Z}},D)\leq \frac{3}{2}\underline{\rm dim}_B(X, d).
\end{align*}

\end{ex}

\begin{proof}
We only show the first inequality since the second one can be obtained by  the similar argument.

Given  $\omega \in \Omega$ and $n \in \mathbb{N}$, denote by $\varphi_n(1,\omega)$    the cardinality  of $1$  appearing  in the first $n$ coordinates  of $\omega$ from $0$ to $n-1$. Let $\varphi_n(2,\omega)=n-\varphi_n(1,\omega)$.  We  rewrite the Bowen metric $D_n^{\omega}$ on $X^{\mathbb{Z}}$  as
$$D_n^{\omega}(x,y)=\max_{0\leq j <n}D(\sigma^{\varphi_j(1,\omega)+2\varphi_j(2,\omega)}(x), \sigma^{\varphi_j(1,\omega)+2\varphi_j(2,\omega)}(y)).$$
Fix  arbitrary a point $y_0 \in X$. Assume that $E$ is a  $(d,\epsilon)$-separated set of $X$ with the  largest cardinality  $s(X,d,\epsilon)$. 
Let $$E_{n,\omega}:=\{(x_m)_{m \in \mathbb{Z}}: x_m\in E ~~\text{for}~ m=\varphi_j(1,\omega)+2\varphi_j(2,\omega), j=0,...,n-1~\text{and}~x_m=y_0 ~\text{for other }~m  \}$$
 Then the set  $E_{n,\omega}$ is a $( D^{\omega}_n,\epsilon)$-separated set of $X^{\mathbb{Z}}$, and hence $s_n(X^{\mathbb{Z}},D,\omega,\epsilon)\geq  s(X,d,\epsilon)^n$. This implies that $$\overline{\rm dim}_B(X, d)\leq \mathbb{E}_{\mathbb{P}}\overline{{\rm mdim}}_M(T,X^{\mathbb{Z}},D).$$

Now fix $\epsilon >0$ and choose a positive integer $n_0$ sufficiently large such that $$\sum_{|n|\geq n_0}\frac{\diam (X,d)}{2^{|n|}}<\frac{\epsilon}{2}.$$
Assume that  $F$ is a  $(d,\frac{\epsilon}{2})$-spanning set of $X$ with the  smallest cardinality  $r(X,d,\frac{\epsilon}{2})$.   Define 
$$F_{n,\omega}:=\{(x_m)_m: x_j\in F ~\text{for}~ -n_0<j \leq\varphi_n(1,\omega)+2\varphi_n(2,\omega)+n_0, ~\text{and}~y_0~\text{for other} ~m\}.$$
It is easy to see that  the set $F_{n,\omega}$ is a $(D_n^\omega,\epsilon)$-spanning set of $X^{\mathbb{Z}}$. Therefore,
\begin{align*}
&\limsup_{n\to \infty}\frac{ 1}{n}\int \log r_{n}(X^{\mathbb{Z}},D,\omega,\epsilon) d \mathbb{P}(\omega)\\
\leq & \limsup_{n\to \infty}\frac{ 1}{n\cdot 2^n}\sum_{\omega \in \{1,2\}^n}(\varphi_n(1,\omega)+2\varphi_n(2,\omega)+2 n_0) \log r(X,d,\frac{\epsilon}{2})\\
= & \log r(X,d,\frac{\epsilon}{2}) \limsup_{n\to \infty}\frac{ 1}{n\cdot 2^n}\sum_{j=0}^{n}\binom{n}{j}(n-j+2j)\\
=&\frac{3}{2} \log r(X,d,\frac{\epsilon}{2}),
\end{align*}
where  $\sum_{j=0}^{n}\binom{n}{j}(n+j)=n 2^n+ n \sum_{j=1}^{n}\binom{n-1}{j-1}=3n\cdot 2^{n-1}$.
Then this inequality allows us to  conclude that
$$\mathbb{E}_{\mathbb{P}}\overline{{\rm mdim}}_M(T,X^{\mathbb{Z}},D)\leq \frac{3}{2}\overline{\rm dim}_B(X, d).$$
\end{proof}

\begin{rem}
Let  $\mathbb{P}_1=\delta_{\omega_1}=\delta_{(1,1.,,,)}$, and  let  $\mathbb{P}_2=\delta_{\omega_2}=\delta_{(2,2.,,,)}$ as in Example \ref{ex 5.4}. One can  show that  \begin{align*}
\mathbb{E}_{\delta_{\omega_1}}\overline{{\rm mdim}}_M(T,X^{\mathbb{Z}},D)&=\overline{\rm dim}_B(X, d),\\
\mathbb{E}_{\delta_{\omega_2}}\overline{{\rm mdim}}_M(T,X^{\mathbb{Z}},D)&=2\cdot \overline{\rm dim}_B(X, d).
\end{align*}
Hence, we emphasize that the random  metric mean dimension not only depends on the compatible metrics on $X$, but also is dependent of the choice of the invariant measure $\mathbb{P}$ on  $\Omega$.
\end{rem}

\subsection{Random average metric mean dimension}

From now on,  we assume that $(\Omega,d_{\Omega})$  in our setting is a  compact metric space such that ${\rm mdim}_M(\theta,\Omega,d_{\Omega})<\infty$.

To  give a positive  answer to  Question 3, using the information of $\Omega$ and $X$,  we introduce  \emph{random average  metric mean dimension} to quantify the level of  disturbance on $\Omega \times X$ caused by the dynamics of $\Omega$ and $X$.   The definition is highly inspired  by the  Kifer's random topological entropy \cite{kif01}, Bufetov's  topological entropy for semi-groups \cite{b99} and the work \cite{rgy25}. Let $$\mathcal{E}:=\{\Omega_{n,\epsilon}\}_{n \in \mathbb{N},\epsilon >0}$$ be a collection of finite subsets of $\Omega$ such that each $\Omega_{n.\epsilon}$ is a $((d_\Omega)_n,\epsilon)$-separated set of $\Omega$ with  the largest cardinality $s(\Omega,(d_{\Omega})_n,\epsilon)$.  For $n\in \mathbb{N}$ and $\epsilon >0$,  we define
\begin{align*}
\overline{s}_n(T,\epsilon):&=\frac{1}{s(\Omega,(d_{\Omega})_n,\epsilon)}\sum_{\omega \in \Omega_{n,\epsilon}} s_n(X,d_X,\omega, \epsilon),\\
\overline{r}_n(T,\epsilon):&=\frac{1}{s(\Omega,(d_{\Omega})_n,\epsilon)}\sum_{\omega \in \Omega_{n,\epsilon}} r_n(X,d_X,\omega, \epsilon),
\end{align*}
and 
\begin{align*}
\overline{s}(T,X,\{d_{\Omega},d_X\}, \mathcal{E},\epsilon)=\limsup_{n\to \infty} \frac{1}{n} \log \overline{s}_{n}(T,\epsilon),\\
\overline{r}(T,X,\{d_{\Omega},d_X\}, \mathcal{E},\epsilon)=\limsup_{n\to \infty} \frac{1}{n} \log \overline{r}_{n}(T,\epsilon).
\end{align*}

\begin{df}\label{df 5.5}
The random  average   upper metric mean dimension of $X$ is defined by 
\begin{align*}
\overline{\rm Amdim}_M(T,X,\{d_{\Omega},d_X\})&=\limsup_{\epsilon \to 0}\frac{1}{\logf}\overline{s}(T,X,\{d_{\Omega},d_X\}, \mathcal{E},\epsilon),\\
&=\limsup_{\epsilon \to 0}\frac{1}{\logf}\overline{r}(T,X,\{d_{\Omega},d_X\}, \mathcal{E},\epsilon).
\end{align*}
\end{df}

Similarly,   replacing  $\limsup_{\epsilon \to 0}$  by   $\liminf_{\epsilon \to 0}$ one can define the  corresponding random  average lower metric mean dimension $\underline{\rm Amdim}_M(G,X,\{d_{\Omega},d_X\})$ of $X$.

Compared with  Definition \ref{df 4.21},  the ``average"  in the definition of   random  metric mean dimension is  the integral of  spanning sets or separated sets of $X$ against an invariant  measure $\mathbb{P}$  on  $\Omega$,  while the ``average" in random average   metric mean dimension  refers to the arithmetical average of spanning sets or separated sets of $X$ over  a certain finite set of $\Omega$.  Such a formulation has a greater advantage in calculating the precise random metric mean dimension.

\begin{thm}[= Theorem \ref{thm 1.4}]\label{thm 5.6}
Under the  above setting, we have the following  Abramov-Rokhlin formulae for random  metric mean dimension:
\begin{align*}
\overline{\rm mdim}_M(G,\Omega \times X, d_{\Omega }\times d_X)&= {\rm mdim}_M(G,\Omega, d_{\Omega})+ \overline{\rm Amdim}_M(G,X,\{d_{\Omega},d_X\},\\
\underline{\rm mdim}_M(G,\Omega \times X, d_{\Omega }\times d_X)&= {\rm mdim}_M(G,\Omega, d_{\Omega})+ \underline{\rm Amdim}_M(G,X,\{d_{\Omega},d_X\}.
\end{align*}

\end{thm}

\begin{proof}
We  solely show the first equality since the second one can be obtained similarly.

Fix a family  $\mathcal{E}:=\{\Omega_{n,\epsilon}\}_{n \in \mathbb{N},\epsilon >0}$ of  the $((d_\Omega)_n,\epsilon)$-separated sets of $\Omega$  with  the largest cardinality $s(\Omega,(d_{\Omega})_n,\epsilon)$. For each $\omega \in \Omega_{n,\epsilon}$ and $n \in \mathbb{N}$, if $E_{n,\omega}$ is a $( d_{n}^\omega,\epsilon)$-separated set of $X$, then  $\bigcup_{\omega\in \Omega_{n,\epsilon}}\bigcup_{x\in E_{n,\omega}}\{(\omega,x)\}$ is a $((d_{\Omega}\times d_X)_{n},\epsilon)$-separated set of $\Omega \times X$. Hence,
\begin{align*}
s(\Omega \times X, (d_{\Omega}\times d_X)_{n},\epsilon)\geq \sum_{\omega\in \Omega_{n,\epsilon}} s_{n}(X,d_X,\omega, \epsilon)
= s(\Omega,(d_{\Omega})_{n},\epsilon)\cdot \overline{s}_{n}(T,\epsilon).
\end{align*}
Using the  inequality: $\limsup_{n \to \infty} (a_n+b_n) \geq \limsup_{n \to\infty} a_n+ \liminf_{n \to\infty}b_n$ for any two  sequences $\{a_n\}$ and $\{b_n\}$ of real numbers, this implies that
\begin{align}\label{equ 5.1}
&\limsup_{n \to \infty}\frac{1}{n}\log s(\Omega \times X, (d_{\Omega}\times d_X)_{n},\epsilon)\\
\geq& \liminf_{n \to \infty}\frac{1}{n}\log s(\Omega, (d_{\Omega})_{n},\epsilon)+ \overline{s}(T,X,\{d_{\Omega},d_X\}, \mathcal{E},\epsilon). \nonumber
\end{align}
In \cite[Proposition 3.4]{yz25}, the authors proved  that $${\rm mdim}_M(\theta,\Omega, d_{\Omega})=\lim_{\epsilon \to 0}\frac{1}{\logf}\liminf_{n \to \infty}\frac{1}{n}\log s(\Omega, (d_{\Omega})_{n},\epsilon).$$
Combining this fact with  (\ref{equ 5.1}), we have
\begin{align*}
\overline{\rm mdim}_M(\Theta,\Omega \times X, d_{\Omega }\times d_X)\geq  {\rm mdim}_M(\theta,\Omega, d_{\Omega})+ \overline{\rm Amdim}_M(T,X,\{d_{\Omega},d_X\}).
\end{align*}

Notice that a $((d_{\Omega})_n,\epsilon)$-separated set of $\Omega$  with the largest cardinality is also a
$((d_{\Omega})_n,\epsilon)$-spanning set of $\Omega$. Then  we have
\begin{align*}
r(\Omega\times X, (d_{\Omega}\times d_X)_{n},\epsilon)\leq  \sum_{\omega\in \Omega_{n,\epsilon}} r_{n}(X,d_X,\omega, \epsilon)
= s(\Omega,(d_{\Omega})_{n},\epsilon)\cdot \overline{r}_{n}(T,\epsilon).
\end{align*}
After taking the corresponding limits, we  similarly obtain that
\begin{align*}
\overline{\rm mdim}_M(\Theta,\Omega \times X, d_{\Omega }\times d_X)\leq  {\rm mdim}_M(\theta,\Omega, d_{\Omega})+ \overline{\rm Amdim}_M(T,X,\{d_{\Omega},d_X\}).
\end{align*}
\end{proof}

\begin{rem}
$(1)$. By  Theorem \ref{thm 5.6}, the random average   upper metric mean dimension of $X$ can be equivalently  given by 
\begin{align*}
\overline{\rm Amdim}_M(T,X,\{d_{\Omega},d_X\})
=\overline{\rm mdim}_M(\Theta,\Omega \times X, d_{\Omega }\times d_X)- {\rm mdim}_M(\theta,\Omega, d_{\Omega}).
\end{align*}
We conclude that  Definition \ref{df 5.5} is well-defined, which does not depend on  the choice of  the family $\mathcal{E}$.

$(2)$.  Since $\pi_{\Omega}: \Omega \times X \rightarrow \Omega$ is one-Lipschtiz,  $h_{top}(\Theta,\Omega\times X)=\infty$  if $h_{top}(\theta,\Omega)=\infty$.  The Abramov-Rokhlin formula of random  metric mean dimension presented in Theorem \ref{thm 5.6} gives  additional  information about the  disturbance on $\Omega \times X$ caused by the dynamics of $\Omega$ and $X$.
\end{rem}

The metric mean dimensions of the driving system, the skew product  and the  non-autonomous dynamical systems  are related by the following inequality.

\begin{cor}\label{cor 5.8} 
Let $(\Omega, d_{\Omega}, \theta)$ be a TDS, and let $T = (T_{n,\omega})$ be a random dynamical system over the measure-preserving system $(\Omega, \mathcal{B}(\Omega), \mathbb{P}, \theta)$. Then 
\begin{align*}
&\overline{\rm mdim}_M(\Theta, \Omega \times X, d_{\Omega }\times d_X)\\
\leq&  {\rm mdim}_M(\theta,\Omega, d_{\Omega})+\limsup_{\epsilon \to 0}\frac{1}{\logf}\limsup_{n\to\infty}\frac{1}{n} \log  \sup_{\omega \in \Omega} s_{n}(X,d_X,\omega, \epsilon).
\end{align*}
 This inequality is also valid for corresponding lower metric mean dimensions by  exchanging  $\limsup_{\epsilon \to 0}$ into $\liminf_{\epsilon \to 0}$,.
\end{cor}

\begin{proof}
It directly follows from  Theorem \ref{thm 5.6} and the fact  $\overline{s}_n(T,\epsilon)\leq  \sup_{\omega \in \Omega} s_{n}(X,d_X,\omega, \epsilon).$  
\end{proof}

If the driving system admits  zero metric mean dimension,  then the  random average  metric mean dimension exactly coincides with the metric mean dimension of skew product.
\begin{cor}\label{cor 5.7}
Let $(\Omega, d_{\Omega}, \theta)$ be a TDS with  zero  upper metric mean dimension, and let $T = (T_{n,\omega})$ be a random dynamical system over the measure-preserving system $(\Omega, \mathcal{B}(\Omega), \mathbb{P}, \theta)$. Then
\begin{align*}
\overline{\rm Amdim}_M(G,X,\{d_{\Omega},d_X\})&=\overline{\rm mdim}_M(\Theta,\Omega \times X, d_{\Omega }\times d_X),\\
\underline{\rm Amdim}_M(G,X,\{d_{\Omega},d_X\})&=\underline{\rm mdim}_M(\Theta,\Omega \times X, d_{\Omega }\times d_X).
\end{align*}
\end{cor}


Finally, we use Theorem \ref{thm 5.6} to  calculate  the  random average  metric mean dimension of   Examples  \ref{ex 5.3} and \ref{ex 5.4}.  It turns out that  the  two  different types of random   metric mean dimensions  behave rather differently  in  terms of the different meaning of the ``average".
\begin{ex}
$(1)$ As in   Example \ref{ex 5.3},  we assume that  ${\rm mdim}_M(\theta,\Omega,d_{\Omega})<\infty$. Notice that for every $n \in \mathbb{N}$, $||x-y||_{n}^\omega=||x-y||_n$ for all $\omega \in \Omega$. Then we have
$$(d_{\Omega}\times ||\cdot||)_{n}((\omega_1,x_1),(\omega_2,x_2))=\max\{(d_{\Omega})_n(\omega_1,\omega_2), ||x_1-x_2||_n\}.$$ 
By  (\ref{equ 2.1}), we have 
$$\overline{\rm mdim}_M(\Theta,\Omega \times X^{\mathbb{Z}}, d_{\Omega}\times ||\cdot||)= {\rm mdim}_M(\theta,\Omega,d_{\Omega})+\overline{\rm dim}_B(X, ||\cdot||_X),$$
where $\Theta(\omega,x)=(\theta \omega, T_{\omega}x)$ is the skew  product transformation. Combining this equality with Theorem \ref{thm 5.6} and Example \ref{ex 5.3}, we have
$$\overline{\rm Amdim}_M(T,X^{\mathbb{Z}},\{d_{\Omega},||\cdot||\})=\overline{\rm dim}_B(X, ||\cdot||_X)=\mathbb{E}_{\mathbb{P}}\overline{{\rm mdim}}_M(T,X^{\mathbb{Z}},||\cdot||).$$

Recall that  in \cite{rj22}  the upper metric mean dimension of the  non-autonomous dynamical system $(X^{\mathbb{Z}}, \{T_{\omega}^n\}_n)$ is given by
$$\overline{\rm mdim}_M(X^{\mathbb{Z}}, \{T_{\omega}^n\}_n, ||\cdot||)=\limsup_{\epsilon \to 0}\limsup_{n \to \infty}\frac{1}{n} \log r(X,||x-y||_{n}^\omega,\epsilon).$$
Using the fact $||x-y||_{n}^\omega=||x-y||_n$ again, for every $\omega \in \Omega$ we have  $\overline{\rm mdim}_M(X^{\mathbb{Z}}, \{T_{\omega}^n\}_n, ||\cdot||)=\overline{\rm dim}_B(X, ||\cdot||_X),$
and hence
$$\overline{\rm mdim}_M(\Theta,\Omega \times X^{\mathbb{Z}}, d_{\Omega }\times ||\cdot||)=\max_{\omega \in \Omega} \limits \overline{\rm mdim}_M(X^{\mathbb{Z}},\{T_{\omega}^n\}_n, ||\cdot||).$$

$(2)$  As in   Example \ref{ex 5.4}, we assume  that $ \overline{\rm dim}_B(X,d)>0$.  Let $$d_{\Omega}(\omega_1,\omega_2)=(\frac{1}{2})^{n(\omega_1,\omega_2)},$$ where $n(\omega_1,\omega_2):=\min\{|n|: \omega_n^1\not=\omega_n^2\}$. The Corollary \ref{cor 5.7} implies that $$\overline{\rm Amdim}_M(T,X^{\mathbb{Z}},\{d_{\Omega},D\})=\overline{\rm mdim}_M(\Theta,\Omega \times X^{\mathbb{Z}}, d_{\Omega }\times D).$$

We claim that  $$\overline{\rm mdim}_M(\Theta,\Omega \times X^{\mathbb{Z}}, d_{\Omega }\times D)= 2\cdot  \overline{ \rm dim}_B(X,d).$$

To see this, consider the constant sequence $\omega=(...,2,2,...)  \in \Omega$.  In this case,  $T_{\omega}^n(x)=(\sigma^2)^{n-1}(x): X^{\mathbb{Z}}\rightarrow X^{\mathbb{Z}}$.  Then $\{\omega\} \times X^{\mathbb{Z}}$ is  a $\Theta$-invariant closed subset of $\Omega \times X^{\mathbb{Z}}$.  Then, by (\ref{equ 2.1}) we have $$\overline{\rm mdim}_M(\Theta,\Omega \times X^{\mathbb{Z}}, d_{\Omega }\times D)\geq  \overline{{\rm mdim}}_M(\sigma^2,X^{\mathbb{Z}},D)=2\cdot \overline{\rm dim}_B(X,d),$$
where the equality is guaranteed by  $\overline{{\rm mdim}}_M(\sigma^2,X^{\mathbb{Z}},D)=2\cdot \overline{{\rm mdim}}_M(\sigma,X^{\mathbb{Z}},D)$ since $ \sigma$ is a Lipschitz map \cite[Proposition 3.9]{yz25}. On the other hand, fix $0<\epsilon <1$ and  assume that $\Omega_{n,\epsilon}$ is a $((d_{\Omega})_n,\epsilon)$-separated set of $\Omega$ with the largest cardinality. Observe that $\varphi_n(1,\omega)+2\varphi_n(2,\omega)\leq 2n$ for all $\omega \in \Omega$.
Similar to the proof of Example \ref{ex 5.4}, one has
\begin{align*}
\overline{r}_n(T,\epsilon)= \frac{1}{s(\Omega,(d_{\Omega})_n,\epsilon)}\sum_{\omega \in \Omega_{n,\epsilon}} r_n(X^{\mathbb{Z}},D,\omega, \epsilon)
\leq   r(X,d,\frac{\epsilon}{2})^{2(n+n_0)},
\end{align*}
which implies that $\overline{\rm Amdim}_M(T,X^{\mathbb{Z}},\{d_{\Omega},D\})\leq 2\cdot \overline{ \rm dim}_B(X,d)$. Hence, the claim is verified. 

Consequently,  we know that $$\mathbb{E}_{\mathbb{P}}\overline{{\rm mdim}}_M(T,X^{\mathbb{Z}},D) \leq \frac{3}{2}\overline{\rm dim}_B(X, d)<\overline{\rm Amdim}_M(T,X^{\mathbb{Z}},\{d_{\Omega},D\}).$$
Furthermore, we  have 
$$\overline{\rm mdim}_M(\Theta,\Omega \times X^{\mathbb{Z}}, d_{\Omega }\times D)=\max_{\omega \in \Omega} \overline{\rm mdim}_M(X^{\mathbb{Z}},\{T_{\omega}^n\}_n, D).$$
\end{ex}

\section*{Acknowledgement} 
 
This work was  supported by  the China Postdoctoral Science Foundation (No. 2024M763856) and  the Postdoctoral Fellowship Program of CPSF  (No. GZC20252040).  




\end{document}